\documentclass{article}
\title{On the Number of Maximum Inner Distance Latin Squares}
\author{by Omar A. Garcia}
\date{December 6th, 2021}

\usepackage[utf8]{inputenc}
\usepackage[english]{babel}
\usepackage{amsthm}
\usepackage{amsmath}
\usepackage{amssymb}
\usepackage[margin=1.0in]{geometry}
\usepackage{graphicx}
\usepackage{float}
\usepackage{mathrsfs}
\usepackage{tikz}
\usepackage{tikz-cd}

\newtheorem{thm}{Theorem}[section]
\newtheorem{lemma}[thm]{Lemma}
\newtheorem{corol}[thm]{Corollary}
\newtheorem*{defn}{Definition}
\newtheorem*{them}{Theorem}
\newtheorem*{conj}{Conjecture}

\newcommand{\func}[3]{\ensuremath{#1 : #2 \to #3}}
\newcommand{\Z}{\mathbb{Z}}

\renewcommand{\phi}{\varphi}
\renewcommand{\emptyset}{\varnothing}
\newcommand{\floor}[1]{\left\lfloor #1 \right\rfloor }

\newcommand{\LSnk}{\mathrm{LS}(n, k)}
\newcommand{\rowprod}{\mathrm{prod}}
\newcommand{\dist}{\mathrm{dist}}
\newcommand{\MiDn}{\mathrm{MiD}(n)}

\begin{document}
\maketitle{}

\begin{abstract}
    The inner distance of a Latin square was defined by myself and six others during an REU in the Summer of 2020 at Moravian College. Since then, I have been curious about its possible connections to other combinatorial mathematics. The inner distance of a matrix is the minimum value of the distance between entries in adjacent cells, where our distance metric is distance modulo $n$. Intuitively, one expects that most Latin squares have inner distance 1, for example there probably exists a pair of adjacent cells with consecutive integers. And very few should have \textit{maximum} inner distance; the maximum inner distance was found by construction for all $n\geq 3$ to be exactly $\floor{\frac{n-1}{2}}$. In this paper we also establish existence for all smaller inner distances. Much of our introductory work is showcased in \cite{inner distance}, with a primary focus on determining the maximum inner distance for Latin squares and Sudoku Latin squares. In this paper, we focus on enumerating maximum inner distance squares, and calculated the exact number for any $n$:
    \begin{them}
    The number of Latin squares of maximum inner distance, denoted $|MiDLS(n)|$, equals:
    \[|MiDLS(n)| = 
    \begin{cases}
    4n & \text{if $n$ is odd}\\
    n(P(n)^2 + 2n) & \text{if $n$ is even}
    \end{cases}
    \]
    Where $P(n)=\frac{1}{4}n^2 + 2$ if $n \equiv 0 \pmod{4}$, and $P(n)=\frac{1}{4}n^2 + 1$ if $n\equiv 2 \pmod{4}$. 
    \end{them}
    \noindent The odd case is easy, shown in \cite{inner distance}. The result for evens is computed in this paper by a connection to graph theory, and direct computation.
\end{abstract}

\section{Vocabulary and Such}
To catch the reader up to speed, all relevant definitions and results are given in the introduction. My goal is to write such that any undergraduate can read the paper and understand every bit of mathematics. 

A \textbf{Latin square} is an $n\times n$ matrix whose entries come from a symbol set $S$ of size $n\in \Z$, where every row and every column contains each symbol in $S$. Equivalently, rows and columns do not contain repeated symbols. More generally a \textbf{Latin rectangle} is a $k\times n$ (or $n\times k$) matrix for $k\leq n$, whose entries come from a symbol set of size $n$, where the rows and columns do not contain repeats. Inner distance will be defined for Latin rectangles. $1\times n$ Latin rectangles are called Latin \textbf{rows}, $n\times 1$ Latin rectangles are called Latin \textbf{columns}, and $n\times n$ Latin rectangles are Latin squares. The symbols $(i,j)$ will denote the $i$-th row and $j$-th column in a matrix. Cells are \textbf{adjacent} if they share an edge vertically or horizontally, for example cells $(1,3)$ and $(2,3)$ share an edge vertically, but cells $(1,2)$ and $(2,3)$ do not share any edge. See the following examples:
\begin{figure} [H]
    \centering
    \begin{tikzpicture}[scale=.6]
    \draw(0,0)grid(6,3); 
    \draw[step=6,ultra thick](0,0)rectangle(6,3);
    \foreach\x[count=\i] in{2, 6, 4, 1, 3, 5}{\node at(\i-0.5,2.5){$\x$};};
    \foreach\x[count=\i] in{5, 2, 6, 4, 1, 3}{\node at(\i-0.5,1.5){$\x$};};
    \foreach\x[count=\i] in{3, 5, 2, 6, 4, 1}{\node at(\i-0.5,0.5){$\x$};};
    \end{tikzpicture}
    \qquad
    \begin{tikzpicture}[scale=.6]
    \draw(0,0)grid(7,1); 
    \draw[step=7,ultra thick](0,0)rectangle(7,1);
    \foreach\x[count=\i] in{4, 2, 5, 6, 1, 3, 7}{\node at(\i-0.5,0.5){$\x$};};
    \end{tikzpicture}
    \caption{A $3\times 6$ Latin rectangle of inner distance 2 (left) and a $1\times 7$ Latin row of inner distance 1 (right).}
    \label{LR examples}
\end{figure}
\begin{figure} [H]
    \centering
    \begin{tikzpicture}[scale=.6]
    \draw(0,0)grid(6,6); 
    \draw[step=6,ultra thick](0,0)grid(6,6);
    \foreach\x[count=\i] in{1, 3, 5, 2, 6, 4}{\node at(\i-0.5,5.5){$\x$};};
    \foreach\x[count=\i] in{4, 1, 3, 5, 2, 6}{\node at(\i-0.5,4.5){$\x$};};
    \foreach\x[count=\i] in{6, 4, 1, 3, 5, 2}{\node at(\i-0.5,3.5){$\x$};};
    \foreach\x[count=\i] in{2, 6, 4, 1, 3, 5}{\node at(\i-0.5,2.5){$\x$};};
    \foreach\x[count=\i] in{5, 2, 6, 4, 1, 3}{\node at(\i-0.5,1.5){$\x$};};
    \foreach\x[count=\i] in{3, 5, 2, 6, 4, 1}{\node at(\i-0.5,0.5){$\x$};};
    \end{tikzpicture}
    \qquad
    \begin{tikzpicture}[scale=.6]
    \draw(0,0)grid(6,6); 
    \draw[step=6,ultra thick](0,0)grid(6,6);
    \foreach\x[count=\i] in{2, 1, 4, 3, 6, 5}{\node at(\i-0.5,5.5){$\x$};};
    \foreach\x[count=\i] in{3, 4, 6, 1, 5, 2}{\node at(\i-0.5,4.5){$\x$};};
    \foreach\x[count=\i] in{1, 6, 3, 5, 2, 4}{\node at(\i-0.5,3.5){$\x$};};
    \foreach\x[count=\i] in{5, 3, 1, 2, 4, 6}{\node at(\i-0.5,2.5){$\x$};};
    \foreach\x[count=\i] in{6, 5, 2, 4, 3, 1}{\node at(\i-0.5,1.5){$\x$};};
    \foreach\x[count=\i] in{4, 2, 5, 6, 1, 3}{\node at(\i-0.5,0.5){$\x$};};
    \end{tikzpicture}
    \caption{Two Latin squares of order 6, of inner distance 2 (left) and 1 (right).}
    \label{LS examples}
\end{figure}

For the rest of the paper, the symbol set of any rectangle will be the integers $1, 2, \dots, n$, denoted $[1, n]$. The symbol `$n$' is an arbitrary positive integer $n\geq5$, since the inner distance of Latin squares of order 4 and below is always 1. In section \ref{even section} we further assume $n\geq 6$ is even. 

To avoid confusion about numbers modulo $n$, we adopt the following conventions: the symbol `$\equiv$' should be interpreted as obtaining the remainder modulo $n$, while the symbol `$=$' means integer equivalence. For example $m_{1,2} - m_{1,1}\equiv h_{1,1}$ means to compute the integer $m_{1,2} - m_{1,1}$, and find its remainder $h_{1,1}$ mod $n$, with $h_{1,1}\in [0,n-1]$. In contrast, $\epsilon_{1,1} = h_{1,1} - \frac{n}{2}$ means $\epsilon_{1,1}$ is the integer $h_{1,1} - \frac{n}{2}$. We use the symbol set $[1,n]$ rather than $[0, n-1]$ for symbols in a Latin square. 

The \textbf{distance} between two symbols $a,b\in [1,n]$ is denoted $\dist(a,b)$, and is the minimum of $a-b \pmod{n}$ and $b-a \pmod{n}$, where $\pmod{n}$ also means to compute the remainder in $[0, n-1]$. For example, if $n=6$, then $\dist(1,6)=1$, since $6+1 \equiv 1$. But if $n\geq 10$, then $\dist(1,6) = 5$, since $1 + 5 \equiv 6$. Intuitively, the distance is the shortest distance from $a$ to $b$.

The \textbf{inner distance} of a Latin rectangle $R$ is the minimum of the distances between all pairs of adjacent cells in $R$. For example in figure \ref{LS examples} the left square has inner distance 2 because all adjacent symbols have a difference of 2,3, or 4, corresponding to distances of 2,3, and 2. Cells $(2,1)$ and $(2,2)$ in the right square contain symbols 3 and 4 respectively, which have a distance of 1. The two rectangles in Figure \ref{LR examples} are sub-rectangles of the squares in Figure \ref{LS examples}, hence have inner distance 2 and 1. 

Enumerating squares of any inner distance is a difficult problem in general: surprisingly the total number of Latin squares of order $n$ is only known up to $n=11$. See \cite{number of LS} for the exact number. Another focus of this paper is formally introducing tools for studying inner distance. Here are a few open questions I would like to see answered: for any fixed order $n$, does the number of squares of inner distance $k$ always decrease as $k$ increases? What structures (algebraic, combinatorial, etc.) correspond to a Latin square of inner distance $k$? What effect does inner distance have on structures induced by Latin squares?

Here is a general roadmap of this paper: section \ref{basic tools} provides all the necessary (new) tools and definitions for understanding the results in the paper. In particular we define difference rows and difference matrices, which are an equivalent way to think about Latin rows (and thus columns) or Latin squares. In section \ref{even section} we specialize to the case where $n$ is even and the inner distance is maximized, and compute how many possible difference rows there are by classifying them all. This gives the number of Latin rows of maximum inner distance. Then by studying this classification, we slowly prove that every Latin square of maximum inner distance is either a circulant or back circulant, or something we call a row product. 

\subsection{Previous Results}
Here we recall some results from our first paper \cite{inner distance} that are relevant to this thesis. 

\begin{thm}\label{MiD of squares}
The maximum inner distance of a Latin square of order $n$ is the largest inner distance that a Latin square of order $n$ can achieve. For any given $n$, this distance is $\floor{\frac{n-1}{2}}$, or $\frac{n-1}{2}$ for odd $n$ and $\frac{n}{2} - 1 = \frac{n-2}{2}$ for even $n$. 
\end{thm}

For odd $n$ this value coincides with the maximum possible distance between any two symbols, making these squares very easy to classify.

\begin{thm}\label{4n odd squares}
For odd $n$, the maximum inner distance of a Latin square is $\frac{n-1}{2}$, and there are exactly $4n$ squares of this inner distance. All squares are of the form $m_{i,j} = s + ri + cj$, where the choice of $r, c\in \{\pm \frac{n-1}{2}\}$ and $s\in[1,n]$ determines a square of maximum inner distance.
\end{thm}

\section{Permutations of Latin Squares and Graph Theory}\label{basic tools}
\subsection{A basic application of symbol permutation.}
From now on, $L$ and $L'$ will denote arbitrary $n\times n$ Latin squares, and the terms $m_{i,j}$ and $m_{i,j}'$ denote the symbol in cells $(i,j)$ of $L$ and $L'$, respectively. 

\begin{defn}
The set of Latin squares of order $n$ and inner distance $k$ is denoted $\mathrm{LS}(n, k)$ for some $k\leq \floor{\frac{n-1}{2}}$. 
\end{defn}
\begin{defn}
Let $L$ be a Latin square of order $n$, and let $\func{\sigma}{S}{S}$ be a permutation (any bijection on $S$). A \textbf{permutation of $L$ by $\sigma$}, is a Latin square $\sigma(L) = M$ where the entry in cell $(i,j)$ of $M$ is $\sigma(m_{i,j})$, and $M$ is said to be a permutation of $L$ if $\sigma$ is unspecified. 
\end{defn}
Permutations are always reversible, and always produce a new Latin squares if the permutation is not the identity map $\sigma(x) = x$. A \textbf{cyclic} permutation $\func{\sigma}{S}{S}$ will be denoted by $\sigma = (x_1, x_2, \dots, x_k)$, which sends $x_i\to x_{i+1}$ for $i\in [1, k-1]$, $x_k \to x_1$, and fixes all other $x\in S$.

\begin{defn}
An \textbf{addition} of a Latin square $L$ is denoted by $L+i$ for some $i\in [1, n]$, as the permutation $\sigma(L)$ where $\sigma(x) = x + i \pmod{n}$. 
\end{defn}
\begin{figure} [H]
    \centering
    \begin{tikzpicture}[scale=.6]
    \draw(0,0)grid(6,6); 
    \draw[step=6,ultra thick](0,0)grid(6,6);
    \foreach\x[count=\i] in{1, 3, 5, 2, 6, 4}{\node at(\i-0.5,5.5){$\x$};};
    \foreach\x[count=\i] in{4, 1, 3, 5, 2, 6}{\node at(\i-0.5,4.5){$\x$};};
    \foreach\x[count=\i] in{6, 4, 1, 3, 5, 2}{\node at(\i-0.5,3.5){$\x$};};
    \foreach\x[count=\i] in{2, 6, 4, 1, 3, 5}{\node at(\i-0.5,2.5){$\x$};};
    \foreach\x[count=\i] in{5, 2, 6, 4, 1, 3}{\node at(\i-0.5,1.5){$\x$};};
    \foreach\x[count=\i] in{3, 5, 2, 6, 4, 1}{\node at(\i-0.5,0.5){$\x$};};
    \end{tikzpicture}
    \qquad
    \begin{tikzpicture}[scale=.6]
    \draw(0,0)grid(6,6); 
    \draw[step=6,ultra thick](0,0)grid(6,6);
    \foreach\x[count=\i] in{3, 5, 2, 6, 4, 1}{\node at(\i-0.5,5.5){$\x$};};
    \foreach\x[count=\i] in{6, 4, 1, 3, 5, 2}{\node at(\i-0.5,4.5){$\x$};};
    \foreach\x[count=\i] in{2, 6, 4, 1, 3, 5}{\node at(\i-0.5,3.5){$\x$};};
    \foreach\x[count=\i] in{4, 1, 3, 5, 2, 6}{\node at(\i-0.5,2.5){$\x$};};
    \foreach\x[count=\i] in{1, 3, 5, 2, 6, 4}{\node at(\i-0.5,1.5){$\x$};};
    \foreach\x[count=\i] in{5, 2, 6, 4, 1, 3}{\node at(\i-0.5,0.5){$\x$};};
    \end{tikzpicture}
    \caption{A Latin square $L$ (left) and its addition $L + 2$ (right). Both have inner distance 2.}
    \label{Addition example}
\end{figure}
\begin{thm}\label{equivalence by adding}
Addition is an equivalence relation on $\LSnk$, with classes of size $n$.
\end{thm}
\begin{proof}
Addition preserves inner distance as it preserves distances between adjacent cells, and so is a well defined operation on $\LSnk$. Let $L\sim M$ if $M$ is an addition of $L$. Then the relation $\sim$ satisfies:
\begin{enumerate}
    \item Reflexive: $L = L+n$ by definition, so $L\sim L$.
    \item Symmetry: If $M = L+i$, then $L = M + (n-i)$, for $i\neq n$, so $M\sim L \iff L \sim M$
    \item Transitive: If $M = L+i$, and $N = M+j$, then $N = L + r$ where $r\equiv i+j$. Therefore $L\sim M$ and $M\sim N$ imply $L\sim N$.
\end{enumerate}
Since $L + i\neq L$ if $i\neq n$, addition provides equivalence classes of $\LSnk$, each having order $n$. Therefore $n$ divides $|\LSnk|$.
\end{proof}
\begin{corol}\label{choice of symbol}
Given any $x\in [1,n]$, and any cell $(i,j)$, there exists a square of inner distance $k$ if and only if there exists a square of inner distance $k$ with $m_{i,j} = x$.
\end{corol}
In particular we may assume $m_{1,1} =1$ within proofs.

\begin{thm}
Suppose $X$ is a set of Latin squares such that symbol permutation is a closed operation on $X$. If $m$ is the maximum inner distance of $X$, there exists squares of inner distance $k$ for all $k\in [1,m]$. 
\end{thm}
\begin{proof}
The proof is constructive. Suppose the maximum inner distance $m$ over $X$ is greater than 1. This means there exists a Latin square $L$ with inner distance $m>1$, and no square in $X$ with inner distance larger than $m$. Within $L$ there exists a pair of adjacent symbols with distance $m$, say $x$ and $x+m \pmod{n}$. Since symbol permutation is a closed operation, by addition of Latin squares we may assume $x = n$, so that $n$ is adjacent to $m$. 

The permutation swap $\sigma = (n, 1)$, creates a square $\sigma(L) = M\in X$. We wish to show the inner distance of $M$ is $m-1$. We do so by checking the distance of every pair of adjacent symbols. 

Suppose $y_1, y_2$ are adjacent symbols in $L$, thus $\dist(y_1, y_2)\geq m$ by assumption. 

1) If $y_1,y_2\not\in \{1, n\}$, then $y_1, y_2$ are unaffected by $\sigma$ and are still adjacent in $M$.

2) We cannot have $\{y_1, y_2\} = \{1, n\}$, otherwise $L$ has inner distance 1, contradicting $m > 1$.

3) Suppose without loss of generality (WLOG), $y_1 = n$ and $y_2\neq 1$. Since $\dist(n, y_2) \geq k$, we have $k\leq y_2 \leq n-k$, and thus $k - 1\leq y_2-1\leq n-k-1$. Therefore $\dist(y_2, 1)\geq k-1$. A similar argument holds if $y_1 = 1$ and $y_2\neq n$.

In particular, a cyclic permutation of consecutive integers $(x, x+1)$ changes the inner distance by at most $\pm 1$. In our example, the symbols 1 and $m$ are adjacent in $M$, so $M$ has inner distance $m-1$. The argument only assumes $m>1$, so the same holds for all $k\in [2,m]$, thus there exists squares in $X$ of any inner distance in $[1, m]$.
\end{proof}
Two Latin squares $L$ and $L'$ are \textbf{isotopic} if $L$ and $L'$ differ by some combination of symbol permutation, row permutations, and column permutations. This definition is only used here:
\begin{corol}
There exists Latin squares of any inner distance up to $\floor{\frac{n-1}{2}}$. For any square $L$ of inner distance $k$, $L$ is isotopic to squares of all inner distances less than $k$, by just symbol permutation. 
\end{corol}

\begin{conj}
Let $m = \floor{\frac{n-1}{2}}$. Then:
\[\left|\mathrm{LS}(n, 1)\right| \geq |\mathrm{LS}(n, 2)|\geq \dots \geq |\mathrm{LS}(n, m)|\]
\end{conj}

\subsection{Graph Theory and Difference Rows}\label{Graph Theory subsection}
A \textbf{graph} is a collection of vertices and edges, usually denoted by a non-empty set of vertices, and set of edges that `connect' the vertices. In this paper, we only need some basic examples and definitions:
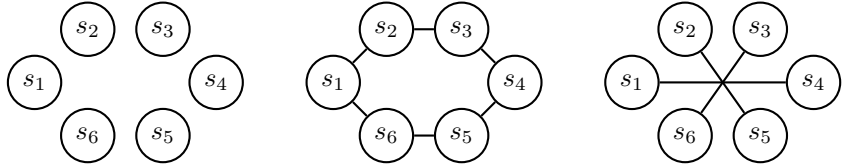
\begin{figure}[H]
    \centering
    \begin{tikzpicture}[thick, main/.style = {draw, circle}] 
        \node[main] (1) {$s_1$};
        \node[main] (2) [above right of=1] {$s_2$};
        \node[main] (3) [right of=2] {$s_3$}; 
        \node[main] (4) [below right of=3] {$s_4$};
        \node[main] (5) [below left of=4] {$s_5$}; 
        \node[main] (6) [left of=5] {$s_6$};
    \end{tikzpicture}
    \qquad
    \begin{tikzpicture}[thick, main/.style = {draw, circle}] 
        \node[main] (1) {$s_1$};
        \node[main] (2) [above right of=1] {$s_2$};
        \node[main] (3) [right of=2] {$s_3$}; 
        \node[main] (4) [below right of=3] {$s_4$};
        \node[main] (5) [below left of=4] {$s_5$}; 
        \node[main] (6) [left of=5] {$s_6$};
        \draw (1) -- (2);
        \draw (2) -- (3);
        \draw (3) -- (4);
        \draw (4) -- (5);
        \draw (5) -- (6);
        \draw (6) -- (1);
    \end{tikzpicture}
    \qquad
    \begin{tikzpicture}[thick, main/.style = {draw, circle}] 
        \node[main] (1) {$s_1$};
        \node[main] (2) [above right of=1] {$s_2$};
        \node[main] (3) [right of=2] {$s_3$}; 
        \node[main] (4) [below right of=3] {$s_4$};
        \node[main] (5) [below left of=4] {$s_5$}; 
        \node[main] (6) [left of=5] {$s_6$};
        \draw (1) -- (4);
        \draw (2) -- (5);
        \draw (3) -- (6);
    \end{tikzpicture}
    \caption{Three examples of graphs using the same basic vertices structure.}
    \label{graph examples}
\end{figure}

Where edges `touch' are not necessarily vertices. The vertices and edges are specified by sets denoted $V, E$. Let $G_1, G_2, G_3$ denote the graphs on the left, center, and right in figure \ref{graph examples}. For all three graphs $V = \{s_1, s_2, s_3, s_4, s_5, s_6\}$; for $G_1$, $E = \emptyset$, while for $G_2$ we can express $E$ as $\{(1,2), (2, 3), (3, 4), (4, 5), (5, 6), (6,1)\}$, and for $G_3$ we have $E = \{(1,4), (2,5), (3,6)\}$. Some basic properties we assume are: each pair of distinct vertices has at most 1 edge between them, and that there are a finite number of vertices. Lastly, graph $G$ is a \textit{subgraph} of graph H if the vertex and edge sets from $G$ are subsets of those from $H$.

The complete graph $K_n$ of order $n$ is the graph with $n$ vertices, and an edge between any pair of distinct vertices. There is a natural connection to Latin rows. Consider the following example:
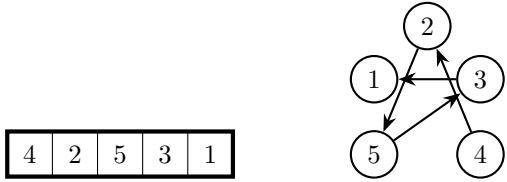
\begin{figure}[H]
    \centering
    \begin{tikzpicture}[scale=.6]
    \draw(0,0)grid(5,1); 
    \draw[step=5,ultra thick](0,0)rectangle(5,1);
    \foreach\x[count=\i] in{4, 2, 5, 3, 1}{\node at(\i-0.5,0.5){$\x$};};
    \end{tikzpicture}
    \qquad\qquad
    \begin{tikzpicture}[node distance = 10mm, thick, main/.style = {draw, circle}] 
        \node[main] (1) {$1$};
        \node[main] (2) [above right of=1] {$2$};
        \node[main] (3) [below right of=2] {$3$}; 
        \node[main] (4) [below of=3] {$4$};
        \node[main] (5) [below of=1] {$5$}; 
        \draw[-{Stealth}] (4) -- (2);
        \draw[-{Stealth}] (2) -- (5);
        \draw[-{Stealth}] (5) -- (3);
        \draw[-{Stealth}] (3) -- (1);
    \end{tikzpicture}
    \caption{A row of maximum inner distance (2 for $n=5$) as well as an \textit{induced} graph.}
    \label{row to graph connection}
\end{figure}

Construct a graph $G$ with vertex set $[1,n]$, and place the vertices in the natural order on a circle. For some Latin row $r = (s_1, s_2, \dots, s_n)$, add an edge $h_i$ between vertices $s_i$ and $s_{i+1}$ for $1\leq i\leq n-1$. This graph $G$ is a subgraph of $K_n$. The vector $(h_1, h_2, \dots, h_{n-1})$ now describes a list of vertices to take along $K_n$, starting at $s_1$ and ending at $s_n$. In contrast, any path that traverses every vertex in $K_n$ exactly once can be described as a Latin row $r$. But the edges to take depend on which vertex one starts on. We formalize these definitions now:
\begin{defn}
A \textbf{Hamiltonian path}, or simply a \textbf{path}, is an ordered list of edges of a graph $G$, $(e_1, e_2, \dots, e_{n-1})$ that `traverses' every vertex in $G$ exactly once, starting at some specified vertex $s_1$. A \textbf{Hamiltonian cycle} is the result of adding the edge back to $s_1$ at the end of a path.
\end{defn}
\begin{defn}
For some $k \leq \floor{\frac{n-1}{2}}$, take $V= \{1, 2, \dots, n\}$ and $E$ the set of all possible edges between symbols $x,y$ in $V$ such that $\dist(x,y) \geq k$. This will be called the \textbf{distance $k$ graph}. The distance $\floor{\frac{n-1}{2}}$ graph is the \textbf{maximum distance graph}, while the distance 1 graph is the complete graph $K_n$.
\end{defn}
\begin{defn}
For any Latin row $r = (s_1, s_2, \dots, s_n)$ with entries in $[1,n]$, the vector $d=(h_1, h_2, \dots, h_{n-1})$ defined by $h_j \equiv s_{j+1} - s_j$ is the \textbf{difference row} $d$ of $r$. A \textbf{normal row} is a Latin row with $s_1 = 1$.
\end{defn}
We can also define \textit{additions} for rows: $r=(s_1, \dots, s_n)$ and $r'=(s_1', \dots, s_n')$, have the same difference row $d$ if and only if for all $j\in [1,n]$, $s_j' \equiv s_j + i$ for some fixed $i\in [1, n]$; that is if $r'$ is an addition of $r$. The proof of the following theorem falls from our discussions above. We can also define additions to rows:
\begin{thm}
A row of inner distance $k$ is equivalent to a Hamiltonian path on the distance $k$ graph. Hamiltonian paths starting at 1 are equivalent to normal rows, which induce unique difference rows. 
\end{thm}

\begin{defn}
A \textbf{circulant} Latin square is a square where row $i+1$ is in the same order as row $i$, with entries shifted to the right by one. A \textbf{back-circulant} Latin square is defined similarly, with entries shifted one to the left. 
\end{defn}
\begin{figure}[H]
    \centering
    \begin{tikzpicture}[scale=.6]
    \draw(0,0)grid(6,6); 
    \draw[step=6,ultra thick](0,0)grid(6,6);
    \foreach\x[count=\i] in{1, 3, 5, 2, 6, 4}{\node at(\i-0.5,5.5){$\x$};};
    \foreach\x[count=\i] in{4, 1, 3, 5, 2, 6}{\node at(\i-0.5,4.5){$\x$};};
    \foreach\x[count=\i] in{6, 4, 1, 3, 5, 2}{\node at(\i-0.5,3.5){$\x$};};
    \foreach\x[count=\i] in{2, 6, 4, 1, 3, 5}{\node at(\i-0.5,2.5){$\x$};};
    \foreach\x[count=\i] in{5, 2, 6, 4, 1, 3}{\node at(\i-0.5,1.5){$\x$};};
    \foreach\x[count=\i] in{3, 5, 2, 6, 4, 1}{\node at(\i-0.5,0.5){$\x$};};
    \end{tikzpicture}
    \qquad
    \begin{tikzpicture}[scale=.6]
    \draw(0,0)grid(7,7); 
    \draw[step=7,ultra thick](0,0)grid(7,7);
    \foreach\x[count=\i] in{1, 2, 3, 4, 5, 6, 7}{\node at(\i-0.5,6.5){$\x$};};
    \foreach\x[count=\i] in{2, 3, 4, 5, 6, 7, 1}{\node at(\i-0.5,5.5){$\x$};};
    \foreach\x[count=\i] in{3, 4, 5, 6, 7, 2, 3}{\node at(\i-0.5,4.5){$\x$};};
    \foreach\x[count=\i] in{4, 5, 6, 7, 2, 3, 4}{\node at(\i-0.5,3.5){$\x$};};
    \foreach\x[count=\i] in{5, 6, 7, 2, 3, 4, 5}{\node at(\i-0.5,2.5){$\x$};};
    \foreach\x[count=\i] in{6, 7, 2, 3, 4, 5, 6}{\node at(\i-0.5,1.5){$\x$};};
    \foreach\x[count=\i] in{7, 2, 3, 4, 5, 6, 1}{\node at(\i-0.5,0.5){$\x$};};
    \end{tikzpicture}
    \caption{An example of a circulant (left) and back-circulant Latin square (right).}
    \label{circulant examples}
\end{figure}
Observe that: only one row is needed to define either kind of circulant, namely the first row is repeated all the way down. Further, each symbol has exactly two neighbors throughout the entire square. As a result:

\begin{thm}\label{circulants equal to cycles}
A circulant Latin square of inner distance $k$ is equivalent to a Hamiltonian cycle on the distance $k$ graph. The same is true for back-circulants. 
\end{thm}
\begin{proof}
A Hamiltonian cycle on the distance $k$ graph corresponds to a Latin row of distance $k$, such that the first and last entries also are distance $k$ apart. If $r = (s_1, s_2, \dots, s_n)$ is the first row in a circulant Latin square $L$, then across all of $L$: $s_j$ is only adjacent to $s_{j-1}$ and $s_{j+1}$ for $2\leq j\leq n-1$, while $s_1$ and $s_n$ are also adjacent to each other. It follows that the choice of Hamiltonian cycle and starting entry $s_1$ determine a unique circulant Latin square, and vice versa. The same argument holds for back-circulants.
\end{proof}

Certain operations on Latin rectangles that preserve inner distance correspond to nice symmetries of Hamiltonian paths on the distance $k$ graph. The difference row $d=(h_1, h_2, \dots, h_{n-1})$ corresponds to a Hamiltonian path starting at 1 on the distance $k$ graph, with vertices placed evenly and in order on a circle. 
\begin{enumerate}
    \item Additions of a row by $i$, corresponds to rotating the edges of the path around the circle by $i$ vertices.
    \item Reversing a difference row $d = (h_1, h_2, \dots, h_{n-1})\to (h_{n-1}, h_{n-2}, \dots, h_1) = d'$ corresponds to finding the Hamiltonian path starting at 1 with difference row $d'$.
    \item Negating a difference row $d \to -d$, where $-d$ sends $h_j\to -h_j\pmod{n}$ corresponds to reflecting the edges around the line from 1 to the origin.
\end{enumerate}
For Hamiltonian cycles, there is another symmetry used in section \ref{classification of cycles}.

\subsection{Difference Matrices}
\begin{defn}
Let $M$ be an $m\times n$ (or $n\times m$) matrix with $1 < m\leq n$ and symbols from $[1, n]$. As usual the symbol in cell $(i,j)$ is $m_{i,j}$. Define matrices $H$ and $V$ as follows: let $h_{i,j}$ and $v_{i,j}$ denote the symbol in cell $(i,j)$ of $H,V$, where $H$ is an $(m-1)\times n$ matrix and $V$ is an $m\times (n-1)$ matrix. Define $h_{i,j} \equiv m_{i,j+1} - m_{i,j}$ for $j\in [1, n-1]$ and $v_{i,j}$ as $v_{i,j} \equiv m_{i+1,j} - m_{i,j}$. $H$ is the \textbf{horizontal difference matrix} of $L$, and $V$ is the \textbf{vertical difference matrix}. $H$ and $V$ are said to be the \textbf{difference matrices} of $M$.
\end{defn}
These matrix analogs of difference rows, if given a starting symbol $m_{1,1}$, also uniquely determine a Latin square $L$, though the conditions for $H,V$ need to be stated precisely:

\begin{lemma}\label{H V conditions}
Suppose $H,V$ are matrices with entries denoted $h_{i,j}$ and $v_{i,j}$ from $[1,n]$, with $H$ an $m\times (n-1)$ matrix and $V$ an $(m-1)\times n$ matrix, for some $2\leq m\leq n$. 
\begin{enumerate}
    \item[0.] A vector $d = (h_1, h_2, \dots, h_{n-1})$ is the difference row of a Latin row if and only if for all $j_1, j_2$ with $1\leq j_1<j_2\leq n-1$, we have $\sum_{j=j_1}^{j_2} h_{j} \not \equiv 0$.
    \item $H$ and $V$ define a unique $m\times n$ matrix with a 1 in the top-left-most entry if and only if for all appropriate $(i,j)$ the entries satisfy
    \[h_{i,j} + v_{i, j+1} \equiv v_{i, j} + h_{i+1, j} \pmod{n}\]
    Assuming Condition 1 holds, $H,V$ define a unique $m\times n$ Latin rectangle if and only if both of the following conditions are met:

    \item For each row $i$, and any $1\leq j_1<j_2\leq n-1$, we have $\sum_{j = j_1}^{j_2} h_{i,j}\not\equiv 0$.
    
    \item For each column $j$, and any $1\leq i_1<i_2\leq m-1$, we have $\sum_{i = i_1}^{i_2} v_{i,j}\not\equiv 0$.
\end{enumerate}
A similar statement holds true if the dimensions of $H,V$ are swapped.
\end{lemma}
\begin{proof}
Statements 0 gives a necessary and sufficient condition for a difference row to determine a unique Latin row: since no symbol in $r$ repeats, the difference between any two symbols in $r$ is non-zero. Statements 2 and 3 follow the same logic, for entries in a Latin rectangle.

Statement 1 is a necessary and sufficient condition for $H,V$ to define a matrix: suppose that all $i,j$ satisfy Condition 1 of $H,V$. Then $m_{1,1} = 1$, and $m_{i,j} = m_{1,1} + \sum_{k=1}^{j-1}v_{k, 1} + \sum_{k=1}^{j-1}h_{i, k}$ define a unique matrix $M$. Now if $M$ has $H,V$ as its difference matrices, then the condition must clearly hold by definition of $h_{i,j}$ and $v_{i,j}$ terms. See the following diagram.
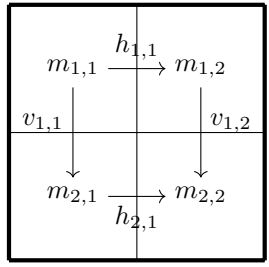
\begin{figure}[H]
    \centering
    \begin{tikzpicture}[scale=1.7]
    \draw(0,0)grid(2,2); 
    \draw[step=2,ultra thick](0,0)grid(2,2);
    \node  at(0.5, 1.5) (1) {$m_{1,1}$};
    \node  at(1.5, 1.5) (2) {$m_{1,2}$};
    \node  at(0.5, 0.5) (3) {$m_{2,1}$};
    \node  at(1.5, 0.5) (4) {$m_{2,2}$};
    \draw[->] (1) -- node[midway, above] {$h_{1,1}$} (2);
    \draw[->] (2) -- node[midway, right, pos=0.4] {$v_{1,2}$} (4);
    \draw[->] (1) -- node[midway, left, pos=0.4] {$v_{1,1}$} (3);
    \draw[->] (3) -- node[midway, below] {$h_{2,1}$} (4);
    \end{tikzpicture}
    \caption{Proof by picture: We must have for any appropriate $i,j$ that $h_{i,j} + v_{i,j+1}\equiv v_{i,j} + h_{i+1, j}\pmod{n}$}
    \label{square condition}
\end{figure}
\end{proof}

The concept of difference matrices, though clunky, allows for easier discussion of operations on Latin squares that preserve inner distance.

\begin{defn}
Let $H$ be an $n\times (n-1)$ matrix and $V$ a $(n-1)\times n$ matrix that together satisfy properties 1-3 from Lemma \ref{H V conditions}. The $n\times n$ matrix induced by $H,V$ with the symbol $s$ in cell $(1,1)$ is denoted $\ell(s, H,V)$.
\end{defn}
\begin{lemma}
Two Latin squares $L$ and $L'$ have the same difference matrices if and only if $L'$ is an addition of $L$. In other words, $L = \ell(s, H, V)$ and $L' = \ell(s', H, V)$ for symbols $s, s'$.
\end{lemma}
\begin{proof}
This is just a generalization of the proof for difference rows.
\end{proof}

The inspiration for studying difference matrices came from looking at the squares of maximum inner distance for odd $n$. There are exactly $4n$ such squares, corresponding to 4 classes of $n$ additions of squares. The 4 classes for $n=5$ are shown below.
\begin{figure}[H]
    \centering
    \begin{tikzpicture}[scale=.6]
    \draw(0,0)grid(5,5); 
    \draw[step=5,ultra thick](0,0)grid(5,5);
    \foreach\x[count=\i] in{1, 3, 5, 2, 4}{\node at(\i-0.5,4.5){$\x$};};
    \foreach\x[count=\i] in{3, 5, 2, 4, 1}{\node at(\i-0.5,3.5){$\x$};};
    \foreach\x[count=\i] in{5, 2, 4, 1, 3}{\node at(\i-0.5,2.5){$\x$};};
    \foreach\x[count=\i] in{2, 4, 1, 3, 5}{\node at(\i-0.5,1.5){$\x$};};
    \foreach\x[count=\i] in{4, 1, 3, 5, 2}{\node at(\i-0.5,0.5){$\x$};};
    \end{tikzpicture}
    \quad
    \begin{tikzpicture}[scale=.6]
    \draw(0,0)grid(5,5); 
    \draw[step=5,ultra thick](0,0)grid(5,5);
    \foreach\x[count=\i] in{1, 3, 5, 2, 4}{\node at(\i-0.5,4.5){$\x$};};
    \foreach\x[count=\i] in{4, 1, 3, 5, 2}{\node at(\i-0.5,3.5){$\x$};};
    \foreach\x[count=\i] in{2, 4, 1, 3, 5}{\node at(\i-0.5,2.5){$\x$};};
    \foreach\x[count=\i] in{5, 2, 4, 1, 3}{\node at(\i-0.5,1.5){$\x$};};
    \foreach\x[count=\i] in{3, 5, 2, 4, 1}{\node at(\i-0.5,0.5){$\x$};};
    \end{tikzpicture}\\\vspace{0.1in}
    \begin{tikzpicture}[scale=.6]
    \draw(0,0)grid(5,5); 
    \draw[step=5,ultra thick](0,0)grid(5,5);
    \foreach\x[count=\i] in{1, 4, 2, 5, 3}{\node at(\i-0.5,4.5){$\x$};};
    \foreach\x[count=\i] in{4, 2, 5, 3, 1}{\node at(\i-0.5,3.5){$\x$};};
    \foreach\x[count=\i] in{2, 5, 3, 1, 4}{\node at(\i-0.5,2.5){$\x$};};
    \foreach\x[count=\i] in{5, 3, 1, 4, 2}{\node at(\i-0.5,1.5){$\x$};};
    \foreach\x[count=\i] in{3, 1, 4, 2, 5}{\node at(\i-0.5,0.5){$\x$};};
    \end{tikzpicture}
    \quad
    \begin{tikzpicture}[scale=.6]
    \draw(0,0)grid(5,5); 
    \draw[step=5,ultra thick](0,0)grid(5,5);
    \foreach\x[count=\i] in{1, 4, 2, 5, 3}{\node at(\i-0.5,4.5){$\x$};};
    \foreach\x[count=\i] in{3, 1, 4, 2, 5}{\node at(\i-0.5,3.5){$\x$};};
    \foreach\x[count=\i] in{5, 3, 1, 4, 2}{\node at(\i-0.5,2.5){$\x$};};
    \foreach\x[count=\i] in{2, 5, 3, 1, 4}{\node at(\i-0.5,1.5){$\x$};};
    \foreach\x[count=\i] in{4, 2, 5, 3, 1}{\node at(\i-0.5,0.5){$\x$};};
    \end{tikzpicture}
    \caption{A representative of each class with respect to addition of all Latin squares of maximum inner distance for $n=5$ (inner distance 2).}
    \label{MiD of order 5}
\end{figure}
As one may check, there are two horizontal and two vertical difference matrices between the four squares, say $H_1, H_2$ and $V_1,V_2$, with the relation that $H_1 = -H_2$ and $V_1 = -V_2$, where $-L$ denotes the operation $m_{i,j}\to n-m_{i,j}$. In other words, from left to right the two matrices on top are represented by $\ell(1, H, V)$ and $\ell(1, H, -V)$ and the two matrices on the bottom are represented by $\ell(1, -H, V)$ and $\ell(1, -H, -V)$. But in general the 4-fold symmetry does not hold, since the operations $H\to H$ and $V\to -V$ do not preserve Condition 1 from Lemma \ref{H V conditions}. However as one may check the operation $\ell(s, H, V)\to \ell(s, -H, -V)$ always gives a new Latin square. I had originally called $\ell(s, -H, -V)$ the \textbf{distance conjugate} of $\ell(s, H, V)$, which improves Theorem \ref{equivalence by adding} to the following (proof omitted):
\begin{thm}
Distance conjugation, the operation $\ell(s, H, V)\to \ell(s, -H, -V)$ preserves inner distance and pairs addition classes of squares in two. Thus $2n$ divides $|\LSnk|$. Further, distance conjugation is a form of symbol permutation.
\end{thm}
As a final remark on this topic, note that $\ell(s, -H, V)$ is the distance conjugate of $\ell(s, H, -V)$, so one is well defined if and only if the other does. If one follows this rabbit hole, one will stumble upon the following important definition:

\begin{defn}
Let $d = (h_1, h_2, \dots, h_{n-1})$ and $d' = (h_1', h_2', \dots, h_{n-1}')$ be difference rows. The \textbf{row product} of $d$ and $d'$, denoted $\rowprod(d, d')$, is defined as $\ell(1, H, V)$, where $H$ is the $n\times(n-1)$ matrix with every row being $d$ and $V$ is the $(n-1)\times n$ matrix with every column being $d'$ as a column vector. 
\end{defn}
Row products are always well defined as Conditions 1-3 of Lemma \ref{H V conditions} easily hold. 
\begin{lemma}\label{row prod equivalent}
Let $L$ be an $n\times n$ Latin square, and $H,V$ its horizontal and vertical difference matrices. If rows $i$ and $i+1$ of $H$ are equal, then row $i$ of $V$ is constant. If columns $j$ and $j+1$ of $V$ are equal, then row $j$ of $H$ is constant. 

$L$ is a row product if and only if every row of $H$ is equal to the first row, which is true if and only if every column of $V$ is equal to the first column.
\end{lemma}
\begin{proof}
For example by `row $i$ of $V$ is constant' we mean that every entry on row $i$ of $V$ is the same. Lemma \ref{H V conditions} Condition 1 gives that $h_{i+1, j} = h_{i,j}\iff v_{i, j} = v_{i, j+1}$, therefore if rows $i$ and $i+1$ of $H$ are equal everywhere, then row $i$ of $V$ has the constant entry $v_{i, 1}$, and if columns $j$ and $j+1$ of $V$ are equal everywhere, then column $j$ of $H$ has constant entry $h_{1, j}$. The equivalent statements to $L$ being a row product follow immediately.
\end{proof}

An important property of row products is that inner distance of $L$ only depends on the entries in $H,V$, which are exactly the values of $d, d'$. 

Looking back at Figure \ref{MiD of order 5} with this definition in mind, the pattern is more clear: there are exactly 2 normal rows of maximum inner distance for any odd $n$, specifically they are induced by $d = (\frac{n-1}{2}, \frac{n-1}{2}, \dots, \frac{n-1}{2})$ and its negative, $-d = (\frac{n+1}{2}, \frac{n+1}{2}, \dots, \frac{n+1}{2})$. The 4 classes of squares are thus the 4 combinations of row products of $d$ and $d'$.

\begin{lemma}\label{circulants cap products}
If $L = \ell(1, H, V)$ is circulant (or back-circulant), then $H$ and $V$ are both circulant (or back-circulant). $L$ is both a circulant (or back-circulant) and a row product if and only if every entry in $H$ is constant and every entry in $V$ is constant.
\end{lemma}
\begin{proof}
For a circulant Latin square $L$, $m_{i,j}= m_{i-1, j-1}\implies h_{i,j}= h_{i-1, j-1}$ and $v_{i,j} = v_{i-1,j-1}$. So that $L$ is circulant if and only if both $H$ and $V$ are. A similar statement holds for back-circulants.

If $L$ is a circulant, the main diagonal of $H$ contains one constant value, $h_{1,1}$. If $L$ is a row product, then column $j$ of $H$ contains one constant value, $h_{1, j}$. Therefore $L$ is both circulant and a row product then every entry in $H$ is the same value, $h_{1,1}$. The same argument shows $V$ has every entry being $v_{1,1}$.
\end{proof}

\section{Enumerating Squares of Maximum Inner Distance}\label{even section}
In this section assume that $n\geq 6$ is an even value. We now lead to the main result of the paper. The first step is to count how many Hamiltonian cycles and Hamiltonian paths there are in the maximal distance graph. If there are $P$ paths and $C$ cycles, then we can construct $P^2$ Latin squares as products of rows and $2C$ circulant and back-circulant Latin squares. Lemma \ref{circulants cap products} governs how these two sets overlap. We then try to prove that \textit{every} Latin square of maximum inner distance is of one of these two forms. Both steps of this process require a great understanding of the $P$ Hamiltonian paths, so we completely classify them.

\subsection{Extended Difference Rows}
\begin{defn}
Let $r = (s_1, s_2, \dots, s_n)$ be a row and let $d = (h_1, h_2, \dots, h_{n-1})$ be its difference row. The \textbf{extended difference row} $d_*$ of $r$ is the vector $d_* = (\epsilon_1, \epsilon_2, \dots, \epsilon_{n-1}, h)$, where $\epsilon_j = h_{j} - \frac{n}{2}$ and $h = h_n - \frac{n}{2}$, where $h_n \equiv s_1 - s_n \pmod{n}$.
\end{defn}

Note that the $\epsilon_j$ terms are not taken mod $n$, so they may be negative. It should be clear that normal rows, difference rows, and extended difference rows all induce each other uniquely. The graph interpretation of $d_*$ is as follows: if the inner distance is $\frac{n}{2} - r$, there are $2r+1$ edges, labeled $-r, -r+1, \dots, r-1, r$, attached to each vertex $x$. The edge $\epsilon$ takes $x$ to $x + \frac{n}{2} + \epsilon$ for each $\epsilon\in [-r, r]$. But unlike difference rows, the values of $\epsilon$ stay small with large inner distance. Below are three examples of rows, as well as their difference and extended difference rows to illustrate this.

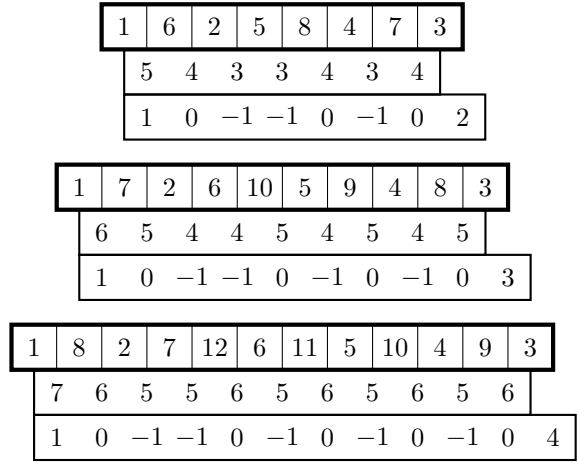
\begin{figure}[H]
    \centering
    \begin{tikzpicture}[scale=.6]
    \draw(0,2)grid(8,3); 
    \draw[step=8,ultra thick](0,2)rectangle(8,3);
    \foreach\x[count=\i] in{1, 6, 2, 5, 8, 4, 7, 3}{\node at(\i-0.5,2.5){$\x$};};
    
    \draw[step=7, thick](0.5,1)rectangle(7.5,2);
    \foreach\x[count=\i] in{5, 4, 3, 3, 4, 3, 4}{\node at(\i,1.5){$\x$};};
    
    \draw[step=8, thick](0.5,0)rectangle(8.5,1);
    \foreach\x[count=\i] in{1, 0, -1, -1, 0, -1, 0, 2}{\node at(\i,0.5){$\x$};};
    \end{tikzpicture}\\\vspace{0.1in}
    
    \begin{tikzpicture}[scale=.6]
    \draw(0,2)grid(10,3); 
    \draw[step=10,ultra thick](0,2)rectangle(10,3);
    \foreach\x[count=\i] in{1, 7, 2, 6, 10, 5, 9, 4, 8, 3}{\node at(\i-0.5,2.5){$\x$};};
    
    \draw[step=9, thick](0.5,1)rectangle(9.5,2);
    \foreach\x[count=\i] in{6, 5, 4, 4, 5, 4, 5, 4, 5}{\node at(\i,1.5){$\x$};};
    
    \draw[step=10, thick](0.5,0)rectangle(10.5,1);
    \foreach\x[count=\i] in{1, 0, -1, -1, 0, -1, 0, -1, 0, 3}{\node at(\i,0.5){$\x$};};
    \end{tikzpicture}\\\vspace{0.1in}
    
    \begin{tikzpicture}[scale=.6]
    \draw(0,2)grid(12,3); 
    \draw[step=12,ultra thick](0,2)rectangle(12,3);
    \foreach\x[count=\i] in{1, 8, 2, 7, 12, 6, 11, 5, 10, 4, 9, 3}{\node at(\i-0.5,2.5){$\x$};};
    
    \draw[step=11, thick](0.5,1)rectangle(11.5,2);
    \foreach\x[count=\i] in{7, 6, 5, 5, 6, 5, 6, 5, 6, 5, 6}{\node at(\i,1.5){$\x$};};
    
    \draw[step=12, thick](0.5,0)rectangle(12.5,1);
    \foreach\x[count=\i] in{1, 0, -1, -1, 0, -1, 0, -1, 0, -1, 0, 4}{\node at(\i,0.5){$\x$};};
    \end{tikzpicture}\\\vspace{0.1in}
    
    \caption{Three triples of a row $r$ (top row) followed by its difference row $d$ (middle row) and extended difference row $d_*$ (bottom row).}
    \label{row, diff row, ext diff row}
\end{figure}
The entries in $d$ and $d_*$ are aligned with the space between symbols in $r$, as they represent differences between these symbols. As $n$ increases, entries in $d$ increase, but the values in $d_*$ stay fixed in magnitude. Figure \ref{row, diff row, ext diff row} highlights several important properties which we leave to the reader to verify.

\begin{lemma}\label{diff row properties}
Let $r, d, d_*$ denote a normal row of inner distance $k$, its induced difference row, and its induced extended difference row, using the notation from the definition. 
\begin{enumerate}
    \item $h + \sum_{j=1}^n \epsilon_j \equiv 0 \pmod{n}$
    \item For each $j\in [1, n-1]$, $|\epsilon_j| \leq \frac{n}{2} - k$. The inner distance $k$ equals $\min\{\dist(h_j, 0)\}_{j=1}^{n-1} = \frac{n}{2} - \max\{|\epsilon_j|\}_{j=1}^{n-1}$.
    \item The symbol $h$ determines the final entry in $r$. $r$ is a Hamiltonian cycle if and only if $|h|\leq k$.
\end{enumerate}
\end{lemma}

The following algorithm brute force counts $P(n)$, the number of normal rows of length $n$ (we use the example $n=6$): Place a 1 in the left-most entry of a vector of length 6. Following 1 is a 3, 4, or 5. If we place a 3, the next symbol can be 5, 6, but not 1 since that was in the previous entry. We repeat this process and check every possibility until we obtain a complete list, which has length 10. This same algorithm works for lower inner distances too, we need only compute the `neighbors mod $n$' of each symbol $x$. The same computations are easy (though a bit time consuming by hand) for $n= 8, 10$, and maybe stop at 12:
\begin{align*}\label{oeis sequences}
    &P(4k+2) = 10, 26, 50, 82, 122, \dots\qquad & \text{for $k = 1, 2, 3, \dots$}\\
    &P(4k) = 6, 18, 38, 66, 102, \dots & \text{for $k = 1, 2, 3,\dots$}\\
    &P(2k) = 6, 10, 18, 26, 38, \dots & \text{for $k = 2, 3, 4, \dots$}
\end{align*}
These are sequences A069894, A005899, and A248800 respectively (top to bottom) on the OEIS\cite{oeis 4k+2}\cite{oeis 4k}\cite{oeis evens}. We prove these sequences hold in general for large $n$, not by counting the rows $r$ themselves, but by counting the possible extended difference rows!

\begin{figure}[H]
    \centering
    \begin{tikzpicture}[scale=.7]
    \draw(0,0)grid(6,10); 
    \foreach\x[count=\i] in{1, 3, 5, 2, 4, 6}{\node at(\i-0.5,9.5){$\x$};};
    \foreach\x[count=\i] in{1, 3, 5, 2, 6, 4}{\node at(\i-0.5,8.5){$\x$};};
    \foreach\x[count=\i] in{1, 3, 6, 4, 2, 5}{\node at(\i-0.5,7.5){$\x$};};
    \foreach\x[count=\i] in{1, 4, 6, 2, 5, 3}{\node at(\i-0.5,6.5){$\x$};};
    \foreach\x[count=\i] in{1, 4, 6, 3, 5, 2}{\node at(\i-0.5,5.5){$\x$};};
    \foreach\x[count=\i] in{1, 4, 2, 5, 3, 6}{\node at(\i-0.5,4.5){$\x$};};
    \foreach\x[count=\i] in{1, 4, 2, 6, 3, 5}{\node at(\i-0.5,3.5){$\x$};};
    \foreach\x[count=\i] in{1, 5, 2, 4, 6, 3}{\node at(\i-0.5,2.5){$\x$};};
    \foreach\x[count=\i] in{1, 5, 3, 6, 2, 4}{\node at(\i-0.5,1.5){$\x$};};
    \foreach\x[count=\i] in{1, 5, 3, 6, 4, 2}{\node at(\i-0.5,0.5){$\x$};};
    \end{tikzpicture}
    \qquad
    \begin{tikzpicture}[scale=.7]
    \draw(0,0)grid(6,10); 
    \foreach\x[count=\i] in{-1, -1, 0, -1, -1, -2}{\node at(\i-0.5,9.5){$\x$};};
    \foreach\x[count=\i] in{-1, -1, 0, 1, 1, 0}{\node at(\i-0.5,8.5){$\x$};};
    \foreach\x[count=\i] in{-1, 0, 1, 1, 0, -1}{\node at(\i-0.5,7.5){$\x$};};
    \foreach\x[count=\i] in{0, -1, -1, 0, 1, 1}{\node at(\i-0.5,6.5){$\x$};};
    \foreach\x[count=\i] in{0, -1, 0, -1, 0, 2}{\node at(\i-0.5,5.5){$\x$};};
    \foreach\x[count=\i] in{0, 1, 0, 1, 0, -2}{\node at(\i-0.5,4.5){$\x$};};
    \foreach\x[count=\i] in{0, 1, 1, 0, -1, -1}{\node at(\i-0.5,3.5){$\x$};};
    \foreach\x[count=\i] in{1, 0, -1, -1, 0, -1}{\node at(\i-0.5,2.5){$\x$};};
    \foreach\x[count=\i] in{1, 1, 0, -1, -1, 0}{\node at(\i-0.5,1.5){$\x$};};
    \foreach\x[count=\i] in{1, 1, 0, 1, 1, 2}{\node at(\i-0.5,0.5){$\x$};};
    \end{tikzpicture}
    \caption{All 10 maximum inner distance rows (left) and their extended difference rows (right) for $n=6$.} 
    \label{n=6 ext diff row examples}
\end{figure}

\subsection{Counting Hamiltonian Cycles of Maximum Inner Distance}
\begin{defn}
Let $r$ be a cyclic row of inner distance $k$, meaning it represents a Hamiltonian cycle on the distance $k$ graph. Let $d_* = (\epsilon_1, \epsilon_2, \dots, \epsilon_n)$ be its extended difference row, where $\epsilon_n = h$. The extended difference row $(\epsilon_n, \epsilon_1, \dots, \epsilon_{n-1})$ defined by shifting its entries right by 1 is called the rotation by 1 of $d_*$. The rotations by $1, 2, \dots, n$ are defined similarly, and are called the \textbf{rotations} of $d_*$.
\end{defn}
This is not to be confused with the symmetry of actually rotating the Hamiltonian path, which corresponds to addition of a constant value to the entries of $r$. Rotations of $d_*$ correspond to the symmetry of `pushing the edges through the vertices,' shown below.

\begin{figure}[H]
    \centering
    \begin{tikzpicture}[scale=0.6]
        \draw(0,1)grid(6,2); 
    \draw[step=6,ultra thick](0,1)rectangle(6,2);
    \foreach\x[count=\i] in{1, 1, 0, -1, -1, 0}{\node at(\i-0.5,1.5){$\x$};};
    \node at(0, 0) {$~$};
    \end{tikzpicture}
    \quad
    \begin{tikzpicture}[thick, main/.style = {draw, circle}] 
        \node[main] (1) {$s_1$};
        \node[main] (2) [above right of=1] {$s_2$};
        \node[main] (3) [right of=2] {$s_3$}; 
        \node[main] (4) [below right of=3] {$s_4$};
        \node[main] (5) [below left of=4] {$s_5$}; 
        \node[main] (6) [left of=5] {$s_6$};
        \draw[-{Stealth}] (1) -> (5);
        \draw[-{Stealth}] (5) -> (3);
        \draw[-{Stealth}] (3) -- (6);
        \draw[-{Stealth}] (6) -- (2);
        \draw[-{Stealth}] (2) -- (4);
        \draw[dotted, -{Stealth}] (4) --  (1);
    \end{tikzpicture}\\\vspace{0.1in}
    \begin{tikzpicture}[scale=0.6]
        \draw(0,1)grid(6,2); 
    \draw[step=6,ultra thick](0,1)rectangle(6,2);
    \foreach\x[count=\i] in{0, 1, 1, 0, -1, -1}{\node at(\i-0.5,1.5){$\x$};};
    \node at(0, 0) {$~$};
    \end{tikzpicture}
    \quad
    \begin{tikzpicture}[thick, main/.style = {draw, circle}] 
        \node[main] (1) {$s_1$};
        \node[main] (2) [above right of=1] {$s_2$};
        \node[main] (3) [right of=2] {$s_3$}; 
        \node[main] (4) [below right of=3] {$s_4$};
        \node[main] (5) [below left of=4] {$s_5$}; 
        \node[main] (6) [left of=5] {$s_6$};
        \draw[-{Stealth}] (1) -- (4);
        \draw[-{Stealth}] (4) -- (2);
        \draw[-{Stealth}] (2) -- (6);
        \draw[-{Stealth}] (6) -- (3);
        \draw[-{Stealth}] (3) -- (5);
        \draw[-{Stealth}, dotted] (5) -- (1);
    \end{tikzpicture}
    \caption{Rotating a difference row corresponds to re-labeling the vertices by the addition $s_n\to s_1$ and taking the path that originally lead to $s_1$ as the start of the new cycle.}
    \label{rotation symmetry}
\end{figure}
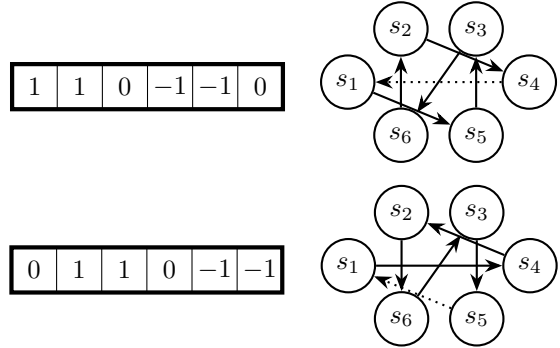
We now count how many Hamiltonian cycles there are on the maximal distance graph by counting extended difference rows. To do so, we study the patterns of 0's, 1's, and -1's within them.

\begin{lemma}\label{disallowed patterns}
Let $d_* = (\epsilon_1, \epsilon_2, \dots, \epsilon_{n-1}, h)$ be an extended difference row of maximum inner distance. The following patterns never show up as subsequences:
\begin{enumerate}
    \item Coming back through the same edge, meaning $(0,0)$, $(1, -1)$, and $(-1, 1)$ are never allowed.
    
    \item Two zeroes bounding $m$ 1's, with $m'$ -1's to either side, \underline{unless} $m' < m$ OR $m=m'=\frac{n}{2}-1$: 
    \[(\underbrace{-1, \dots, -1}_{m_1'}, 0, \underbrace{1, \dots, 1}_{m}, 0, \underbrace{-1, \dots, -1}_{m_2'}) \implies m > m' = m_1' + m_2'\]
    
    \item $q$ sequences of 1's with $q-1$ 0's in between, containing at least $\frac{n}{2}$ total 1's, \underline{unless} (1) $n = 4k$ and $q$ is odd, or (2) $n=4k+2$ and $q$ is even.
    \[(\underbrace{1, \dots, 1}_{m_1}, 0, \underbrace{1, \dots, 1}_{m_2}, 0, \dots, 0, \underbrace{1, \dots, 1}_{m_q}) \qquad \text{Where $\sum_{i=1}^q m_i = \frac{n}{2}$.}\]
    Further, if a zero is attached to the front or back, the sequence is never allowed:
    \[(0, \underbrace{1, \dots, 1}_{m_1}, 0, \underbrace{1, \dots, 1}_{m_2}, 0, \dots, 0, \underbrace{1, \dots, 1}_{m_q}) \qquad (\underbrace{1, \dots, 1}_{m_1}, 0, \underbrace{1, \dots, 1}_{m_2}, 0, \dots, 0, \underbrace{1, \dots, 1}_{m_q}, 0)\]
\end{enumerate}
\end{lemma}
\begin{proof}
We will call these Rules 1-3. Rule 1 is obvious: $\epsilon_j$ correspond to taking the edge that `adds' $\frac{n}{2} + \epsilon_j$. Therefore, two zeroes in a row, or a 1 following a -1 (or vice versa) correspond to travelling back the same path. Rules 2 and 3 are a generalization: the patterns shown amount to a net difference of 0 between the start and the end of the subsequence, which is only possible if the subsequence is the whole row.

Rule 2: Suppose that $m'$ -1's surround two zeroes and $m$ 1's. If $m'\geq m$, then take a sum of horizontal differences $h_j$ between two indices $j_1, j_2$ that contain $m$ many -1's. This amounts to a difference of $m(\frac{n}{2} - 1) + 2\frac{n}{2} + m(\frac{n}{2} + 1) \equiv 0$, computed by adding the number of -1's, 0's, and 1's. By Lemma \ref{diff row properties}, this is only true if $j_1 = 1$ and $j_2 = n$, so then $m = m' = \frac{n}{2} - 1$. 

Rule 3: Suppose a subsequence has $\frac{n}{2}$ 1's with $q-1$ zeroes between them. The partial sum of $h_j$ terms corresponding to this set of indices depends only on the parity of $\frac{n}{2}$ and $q$:
\[\sum h_{j} \equiv \frac{n}{2}\left(\frac{n}{2} + 1\right) + (q-1)\frac{n}{2} = \frac{n}{2}\left(\frac{n}{2} + q -1\right)\]
This evaluates to zero when $(\frac{n}{2} + q -1)$ is even, which is true if $n =4k$ and $q$ is odd or if $n = 4k+2$ and $q$ even. Adding a 0 to the outside of this pattern means adding an extra $\frac{n}{2}$ to the partial sum, inducing a difference of zero in the remaining cases, so such a pattern always fails. 
\end{proof}
Now that we have some basic tools to rule out numerous patterns, we can prove the first result:

\begin{thm}\label{classification of cycles}
Suppose $d_*$ is an extended difference row corresponding to a cycle. Then $d_*$ is a rotation of:
\[(\underbrace{1, \dots, 1}_{\frac{n}{2}-1}, 0, \underbrace{-1, \dots, -1}_{\frac{n}{2}-1}, 0)\]
or $d_*=\pm(1, 1, \dots, 1)$. If $n = 4k$, there are $n+2$ (normal) Hamiltonian cycles on the maximum distance graph. If $n = 4k+2$, there are $n$ (normal) Hamiltonian cycles on the maximum distance graph.
\end{thm}
\begin{proof}
\textit{Step 1}: There is no extended difference row without 0's unless $n=4k$: Since 1's and -1's don't follow each other, the only extended difference rows with no zeroes are $(1,1,\dots, 1)$ and $(-1, -1, \dots, -1)$. Lemma \ref{disallowed patterns} Rule 2 gives $n=4k$. As one may check, both of these difference rows define valid rows. 

\noindent\textit{Step 2}: Suppose then $d_*$ contains 0. WLOG by rotation and negation, $d_*$ begins in $(0, 1, \dots)$. Suppose $d_*$ contains only 0's and 1's, say $q$ sequences of 1's: 
\[(0, \underbrace{1, \dots, 1}_{m_1}, 0, \underbrace{1, \dots, 1}_{m_2}, 0, \dots, 0, \underbrace{1, \dots, 1}_{m_q}, \epsilon_{n-1}, h)\]
Lemma \ref{disallowed patterns} Rule 3 now says $m = \sum_{i=1}^q m_i \leq \frac{n}{2}-1$. Therefore the number of 0's must be at least $\frac{n}{2}$. If we want no 0's to be adjacent, then we must have the following row, which alternates between 0 and 1:
\[(0, 1, 0, 1, \dots, 1, 0, h)\]
Though this is a path, this is not a cycle as this implies $h = \frac{n}{2}-1$. Therefore $d_*$ contains a -1. This implies there exists a subsequence of the form:
\[(0, \underbrace{1, \dots, 1}_{m}, 0, \underbrace{-1, \dots, -1}_{m'})\]
Either this is the end of $d_*$, or following the $m'$ -1's is a 0. In the latter case, Lemma \ref{disallowed patterns} Rule 2 gives that $m=m' = \frac{n}{2} - 1$, but then we have 3 zeroes and $n-2$ non-zero entries in $d_*$, a contradiction. Therefore at the end of $d_*$ is all -1's. But now rotating $d_*$ by -1 implies that it must start in all 1's. Therefore Lemma \eqref{disallowed patterns} Rule 2 says $m = m' = \frac{n}{2}-1$.
\[d_* = (0, \underbrace{1, \dots, 1}_{\frac{n}{2} - 1}, 0, \underbrace{-1, \dots, -1}_{\frac{n}{2} - 1})\]
Therefore every cyclic extended difference row (that contains 1's and -1's) is a rotation of $d_*$ shown above (note that rotation by $\frac{n}{2}$ is the same as negation!), of which there are exactly $n$ of. I leave to the reader to verify this pattern always defines a cyclic row.
\end{proof}
\begin{corol}
If $n=4k$, there are $n + 2$ cirulant Latin squares of maximum inner distance, and $n+2$ back-circulant Latin squares of maximum inner distance. If $n = 4k+2$, there are $n$ of each type of square.
\end{corol}
\begin{proof}
Theorem \ref{circulants equal to cycles} establishes circulants and back-circulants are in bijection to Hamiltonian cycles, which are counted in Theorem \ref{classification of cycles}.
\end{proof}
\begin{corol}\label{MiD circulants cap products}
For $n=4k+2$ there are no circulants or back-circulants of maximum inner distance that are row products. For $n=4k$, there are 2 (normal) circulants and 2 (normal) back-circulants that are row products.
\end{corol}
\begin{proof}
Lemma \ref{circulants cap products} establishes the only squares that are circulants and row products have $H,V$ constant matrices. Thus their difference rows are constant rows, which is only true for $n=4k$ with the difference rows $\pm d$ with $d=(1, 1, \dots, 1)$. The 4 row products from combinations of $d$ and $-d$ are those squares.
\end{proof}

\subsection{Counting Hamiltonian Paths of Maximum Inner Distance}\label{paths subsection}
We classify Hamiltonian paths using the same method. Up to negation there is a unique extended difference row with $h=0$, as we saw in theorem \ref{classification of cycles}. There turn out to be two more classes of paths in general, which we call type 1 and type 2 paths. 

\begin{thm}[Classification of the Hamiltonian Paths]\label{classification of paths}
Let $d_* = (\epsilon_1, \epsilon_2, \dots, \epsilon_{n-1}, h)$ be an extened difference row with $h\geq 0$. Then $d_*$ is of one of the following forms
\begin{enumerate}
    \item[0.] If $h = 0$, $d_*$ is the form of a Hamiltonian cycle in Theorem \ref{classification of cycles}. 
    \item[Type 1.] If $h > 0$ and $d_*$ has only 0's and -1's, then $h = \frac{n}{2} - 1$ and the first $n-1$ entries alternate:
        \[d_* = ( \underbrace{0, -1, 0, -1, \dots, -1, 0}_{\text{alternating}} , \frac{n}{2} - 1)\]
    If $h> 0$ and $d_*$ has only 0's and 1's then $d_*$ is of form:
        \[(\underbrace{1, \dots, 1}_{\frac{n}{2} - h + 1}, \underbrace{0, 1, \dots, 1, 0, 1}_{h-1 \text{ $(0,1)$ pairs}}, \underbrace{1, \dots, 1}_{\frac{n}{2} - h}, h)\]
    The center sequence alternates between 0 and 1, with $h-1$ of each value, for odd $h\in [1, \frac{n}{2}-1]$ if $n = 4k$, and even $h\in [1, \frac{n}{2}-1]$ if $n = 4k+2$. 
    
    \item[Type 2.] If $h > 0$, and contains 0's, 1's, and -1's, then $d_*$ is of form:
        \[(\underbrace{1, \dots, 1}_{m}, 0, \underbrace{-1, \dots, -1}_{m+1}, \underbrace{0, -1, \dots, 0, -1}_{h-1\text{ $(0,-1)$ pairs}}, \underbrace{-1, \dots, -1}_{\frac{n}{2} - h - m - 1}, 0, \underbrace{1, \dots, 1}_{\frac{n}{2} -h -m - 1}, h)\]
    For some $h\in [1, \frac{n}{2}-2]$, and each $m\in [0, \frac{n}{2} -1 - h]$.
\end{enumerate}
\end{thm}
{
\noindent\textbf{Remark.}
By negation, the classification similarly holds for $h<0$ as well, using $|h|$ to count in lieu of $h$. Where not indicated, the ellipses $\dots$ consist only of 1's or only of -1's. By `$h-1~ (0,\pm 1)$ pairs', we mean $h-1$ consecutive pairs of 0 and $\pm1$. The Hamiltonian cycles ($h=\pm1$) are all type 2, with the alternating 0's and -1's vanishing, except for $n=4k$ and $d_* = (1, \dots, 1)$, which is type 1!
}
\begin{proof}
Point 0. is true by the classification of cycles, theorem \ref{classification of cycles}.

{
\noindent\textit{Step 1}: We check that if $d_*$ has no -1's, or no 1's, then it is type 1. In the proof of theorem \ref{classification of cycles}, we also showed that $d_*$ begins in 0 and contains no -1's, then $d_*$ is either the alternating row shown in the theorem statement, or $d_*$ begins in 1's. Supposing the latter is true, let's suppose there are $q$ sequences of 1's. One might recall that reversing a difference row is a valid operation, therefore $d_*$ begins and ends in 1, and thus has $q-1$ zeroes in between the 1's:
\[d_* = (\underbrace{1, \dots, 1}_{m_1}, 0, \underbrace{1, \dots, 1}_{m_2}, 0, \dots, 0, \underbrace{1, \dots, 1}_{m_q}, h)\]
Lemma \ref{disallowed patterns} Rule 3 says $n=4k$ and $q$ is odd, or $n=4k+2$ and $q$ is even. Let $m = \sum_{i=1}^q m_i = \frac{n}{2} + r$ denote the number of 1's in the sequence. Lemma \ref{disallowed patterns} Rule 3 gives $\sum_{i=2}^q m_i < \frac{n}{2}$, therefore $m_1\geq r + 1$, and similarly $\sum_{i=1}^{q-1} m_i < \frac{n}{2}$, therefore $m_q\geq r + 1$. There are $m$ 1's and $q-1$ zeroes in the first $n-1$ entries, so $m + q = n$. Using this, we compute $h$ in terms of $q$:
\[h_n \equiv -\sum_{j=1}^{n-1}h_j \equiv -\left[m\left(\frac{n}{2}+1 \right) + (q-1)\frac{n}{2}\right] \equiv -\left[(n-q)\left(\frac{n}{2}+1 \right) + (q-1)\frac{n}{2}\right] \equiv q + \frac{n}{2}\]
By assumption $1\leq q \leq \frac{n}{2}$, now since $h_n\neq 0$, we strengthen to $q < \frac{n}{2}$. Therefore, $h = h_n - \frac{n}{2} = q$. Further $m+q = \frac{n}{2} + r +q = n\implies r = \frac{n}{2} - h$. By assumption $m_i\geq 1$ for all $i$, and we know that $m_1, m_q \geq r+1 = \frac{n}{2} - h + 1$. Using $n=m +q = m+h$, we show $d_*$ takes the form given by the theorem:
\[n = m +h = m_1 + m_2 + \dots + m_{h-1} + m_h + h \geq \left(\frac{n}{2} - h + 1\right) + \underbrace{1  + \dots + 1}_{h-2} + \left(\frac{n}{2} - h + 1\right) + h = n\]
\[\implies m_1=m_q = \frac{n}{2} - h - 1,\quad m_i = 2 \text{ if }i\in[2,h-1]\]
}

{
\noindent\textit{Step 2}: Now assume the first $n-1$ entries in $d_*$ contain every value in $\{0, \pm 1\}$. The following can easily be derived from Lemma \ref{disallowed patterns} Rule 2:

\textit{There is no subsequence in an extended difference row $d_*$ that begins and ends in 0, and contains both 1's and -1's in between}.

Because 1's and -1's meet with 0's in between, this means that $d_*$ flips signs at most 2 times:
\begin{equation}\label{alt twice}
  d_* = \pm (\underbrace{1, \dots, 1}_{\text{all 1's}}, \underbrace{ 0, -1, \dots, -1, 0}_{\text{0's and -1's}}, \underbrace{1, \dots, 1}_{\text{all 1's}}, h)  
\end{equation}
If $d_*$ alternates just once, it is of one of the forms:
\begin{equation}\label{alt once 1}
    d_* = \pm(\underbrace{1, \dots, 1}_{\text{all 1's}}, \underbrace{ 0, -1, \dots, \epsilon_{n-1}}_{\text{0's and -1's}}, h)
\end{equation}
\begin{equation}\label{alt once 2}
    d_* = \pm(\underbrace{ \epsilon_1, \dots, -1, 0}_{\text{0's and -1's}}, \underbrace{1, \dots, 1}_{\text{all 1's}}, h)
\end{equation}
First assume that $d_*$ is in the form of Equation \ref{alt twice}:
\[(\underbrace{1, \dots, 1}_{m_1}, 0, \underbrace{-1, \dots, -1}_{m_1'}, 0, -1, \dots, -1, 0, \underbrace{-1, \dots, -1}_{m_q'}, 0 \underbrace{1, \dots, 1}_{m_1}, h)\]
Let $q$ denote the number of sequences of -1's in the center, the $i$-th sequence having $m_{i}'>0$ many -1's, and $m_1, m_2>0$ denote the number of 1's on the left and right ends, respectively. Let $m$ and $m'$ denote the total number of 1's and -1's, respectively, in the first $n-1$ entries. Lemma \ref{disallowed patterns} Rule 3 gives $m' < \frac{n}{2}$, and Rule 2 gives $m_1 < m_1'$ and $m_2 < m_q'$. Thus $m' \geq m_1' + m_q' > m$ if $q\geq 2$. Counting all our terms we have $n-1 = m + m' + q+1$ (there are $q+1$ zeroes). First we compute the value of $h_n$ modulo $n$:
\[-h_n \equiv (q+1)\left(\frac{n}{2}\right) + m\left(\frac{n}{2} + 1\right) + m'\left(\frac{n}{2} - 1\right) \equiv (m+m' + q+ 1)\frac{n}{2} + m - m'\pmod{n}\]
\[\implies h_n \equiv m' - m + \frac{n}{2}\pmod{n}\]
If $q=1$: Lemma \ref{disallowed patterns} Rule 2 says exactly that $m = m' = \frac{n}{2} - 1$ and $d_*$ is a Hamiltonian cycle (which is type 2 as remarked above) with $h=1$. 

Now we can assume $h,q > 1$. $q>1$ gives $h_n = m' - m + \frac{n}{2} > \frac{n}{2} \implies h = m' - m$. Since $\frac{n}{2} > m'$ by Rule 3, the relation $m' - m = h$ also implies $\frac{n}{2} - h > m$. Therefore:
\[n = m' + m + q + 2 < \left(\frac{n}{2} - 1 \right) + \left(\frac{n}{2} - h - 1 \right) + q + 2 = n -h + q + 1 \iff h < q+1\]
\begin{equation}\label{ineq1}
    \implies h \leq q
\end{equation}

So $q \geq h\geq 2$. Now using $m + h = m'$ and expanding $m, m'$:
\[m_1 + m_2 + h = m_1' + m_q' +\sum_{i=2}^{q-1}m_i' \geq (m_1 + 1) + (m_2 + 1) + (q-2) = m_1 + m_2 + q \]
\begin{equation}\label{ineq2}
    \implies h\geq q
\end{equation}
Therefore, $h = q$ by equations \ref{ineq1} and \ref{ineq2}. Further, it must also be true that $m_1' = m_1 + 1$ and $m_q' = m_2 + 1$, as well as $m_i' = 1$ for every $i$ in between. We are done with (this case of) the theorem, as now:
\[m' + m + q + 2 = n \leq \left(\frac{n}{2} - 1\right) + \left(\frac{n}{2} -h + 1\right) + h + 2 = n \]
And $n \leq n$ holds if and only if the inequalities $m' \leq \frac{n}{2} - 1$, and $m \leq \frac{n}{2} - h - 1$ are equalities. To summarize, in the first $n-1$ entries there are $(\frac{n}{2} - h -1)$ 1's, $(h + 1)$ 0's, and $(\frac{n}{2} - 1)$ -1's, in the exact order given in the theorem. One might note that setting $h = \frac{n}{2} - 1$ for a type 2 does not work, as the center is too wide. So $h\leq \frac{n}{2}-2$. 

If $\epsilon_{n-1}=0$ in Equation \ref{alt once 1}, we reduce to the case $m_2 = 0$ of Equation \ref{alt twice}. The same happens if $\epsilon_1 = 0$ for Equation \ref{alt once 2}. We show this is true for Equation \ref{alt once 1}, the other I leave to the reader. Assume $d_*$ looks like:
\[d_* = (\underbrace{1, \dots, 1}_{\text{all 1's}}, \underbrace{ 0, -1, \dots, \epsilon_{n-1}}_{\text{0's and -1's}}, h)\]
We wish to show $\epsilon_{n-1} = 0$ if there is more than one sequence of -1's, otherwise we have a cycle. For the sake of contradiction, assume $\epsilon_{n-1} = -1$. Let $m$ denote the number of 1's at the front, and suppose there are $q\geq1$ sequences of -1's, with the $i$-th sequence having $m_i'$ such -1's, with $m' = m_1' + \dots + m_q'$. If $q=1$, Lemma \ref{disallowed patterns} Rules 2 and 3 combine to show that $d_*$ is the Hamiltonian cycle discussed above, so this case is done. Our basic assumptions are now $q > 1$, $m < m' < \frac{n}{2}$ (by Lemma \ref{disallowed patterns} Rule 3) and the sum of terms gives $n -1 = m + m' +q$. Computing $h$ as we did above gives $h = m' - m$. Now setting $m + h = m'$, and expanding $m' = m_1' + \dots + m_q' \geq (m+1) + 1 + \dots + 1 = m + q $ says that $h\geq q$. However using $n = m' + (h-m') + q + 1$, we can use $m'\leq \frac{n}{2} - 1$ to obtain that $h+1 \leq q$. But now, the statements $h\geq q$ and $q \geq h+1$ are in contradiction. 
}
\end{proof}
We have not proven that each of the forms given above actually produce rows of maximum inner distance. One needs only to check that $\sum_{j=j_1}^{j_2}h_{j}$ is non-zero for all $j_1 \leq j_2\leq n-1$, which I leave to the reader.

\begin{thm}\label{number of paths}
Let $n$ be even, and let $P(n)$ denote the number of Hamiltonian paths in the maximum distance graph. If $n=4k$, then $P(n) = \frac{1}{4}n^2 + 2 = 4k^2 + 2$. If $n=4k+2$, then $P(n) = \frac{1}{4}n^2 + 1 = 4k^2 + 4k+2$.
\end{thm}
\begin{proof}
We count how many of each type there are. By Theorem \ref{classification of cycles}, there are 2 paths with $h=0$ for all $n$. 

\textbf{Suppose} $\mathbf{n = 4k}$: Note there is a unique type 1 row for each possible odd value of $h$ in $[1, \frac{n}{2} - 1]$, plus the additional row with $h=\frac{n}{2}-1$, $(0, -1, 0, -1, \dots, -1, 0, \frac{n}{2} - 1)$. So there are $k+1$ type 1's. Type 2's are determined by two values, which we call $(m,h)$. $h$ ranges in $[1, \frac{n}{2} -2]$, and is the last entry in the extended difference row, while $m$ denotes the number of lead 1's, ranging in $[0, \frac{n}{2} - 1 - h]$. So there are $(\frac{n}{2} - h)$ rows ending in $h$, giving a total number of Hamiltonian paths of:
\[P(n) = 2 + 2(\text{(type 1) + (type 2)}) = 2 + 2\left(\frac{n^2}{8} - \frac{n}{4} - 1 + \frac{n}{4} + 1 \right) = \frac{n^2}{4} + 2\]
The 2 coefficient comes from the fact that we need to also count those with $h < 0$. A similar calculation is shown for $n=4k+2$:
\[(\text{type 1's}) = k + 1 = \frac{n+2}{4}\]
\[(\text{type 2's}) = \sum_{h=1}^{\frac{n}{2} - 2}\frac{n}{2} - h = \frac{n^2}{8} - \frac{n}{4} - 1\]
\[P(n) = 2 + 2\left[ \frac{n+2}{4} + \frac{n^2}{8} - \frac{n}{4} - 1 \right] = 2 + 2\left(\frac{n^2}{8} - \frac{1}{2}\right) = \frac{n^2}{4} + 1\]
\end{proof}

\subsection{Counting Latin Squares of Maximum Inner Distance}
For the rest of the paper, the set of Latin squares of maximum inner distance $\floor{\frac{n-1}{2}}$ will be denoted $\MiDn$. Assume $L\in \MiDn$ and as usual $H,V$ denote its horizontal and vertical difference matrices. We will overload some notation: $d_1, d_2, \dots, d_n$ denote both the extended difference rows of rows $1, 2, \dots, n$, and rows $1, 2, \dots, n$ of the horizontal difference matrix $H$, which should be clear in context. $d_i = (\epsilon_{i, 1}, \epsilon_{i, 2}, \dots, \epsilon_{i, n})$, with $\epsilon_{i,n}$ denoting the value previously denoted $h$.

There is one more piece to the puzzle: we prove every even ordered Latin square of maximum inner distance is either a product of rows, a circulant, or a back-circulant. We show this is by using the Classification Theorem \ref{classification of paths} to study how rows behave as neighbors in a Latin square of maximum inner distance. 
\begin{defn}
Let $d_*$ and $d_*'$ be extended difference rows of maximum inner distance. $d_*$ and $d_*'$ are \textbf{neighbors} if their induced rows $r$ and $r'$ may form a $2\times n$ Latin rectangle with of maximum inner distance.
\end{defn}
Occasionally we may also describe $r$ and $r'$ as neighbors. For example, in a row product $\prod(d_*, d_*')$ every row has extended difference row $d_*$. So $d_*$ is its own neighbor. The following rows we will call by name since they are frequently referenced:
\begin{enumerate}
    \item The type 1 row with $h=1-\frac{n}{2}$ that \textbf{alternates} between in the first $n-1$ entries, $(0, 1, 0, \dots, 1, 0, 1-\frac{n}{2})$ will be known as \textbf{Row A}.
    \item The \textbf{cycle} with $h=0$ beginning at 1, $(1, \dots, 1, 0, -1, \dots, -1, 0)$ will be called \textbf{Row C}.
\end{enumerate}
Much of the theorems in this section are based on the following lemma, relating the horizontal differences within neighbors:

\begin{lemma}\label{adj sum lemma}
Let $d_* = (\epsilon_1, \epsilon_2, \dots, \epsilon_{n-1}, h)$ and $d_*' = (\epsilon_1', \epsilon_2', \dots, \epsilon_{n-1}', h')$ be non-equal neighbors, inducing neighboring rows $r = (s_1, s_2, \dots, s_n)$ and $r' = (s_1', s_2', \dots, s_n')$. Then for all $j_1, j_2$, with $1\leq j_1\leq j_2 \leq n-1$:
\[\left|\sum_{j=j_1}^{j_2} (\epsilon_j' - \epsilon_j)\right| \leq 2\]
Further, $|h - h'| \leq 2$ or $|h| = \frac{n}{2} - 1$ and $h' = -h$.
\end{lemma}
\begin{proof}
The proof falls immediately from Lemma \ref{H V conditions} Condition 1:
\[h_{i,j} + v_{i, j+1} \equiv v_{i,j} + h_{i+1, j} \pmod{n}\implies v_{i, j+1} - v_{i,j} \equiv h_{i+1, j} - h_{i,j}\pmod{n}\]
Which then generalizes to the following:
\begin{equation}\label{generalized condition 1}
    v_{i, j_1} - v_{i, j_2} \equiv \sum_{j = j_1}^{j_2-1}(h_{i,j} - h_{i+1, j}) = \sum_{j = j_1}^{j_2-1}(\epsilon_{i,j} - \epsilon_{i+1, j}) \pmod{n}
\end{equation}
Where here and from now on $\epsilon_{i,j} = h_{i,j} - \frac{n}{2}$. In our case $i\in[1,2]$. Assuming $d_*$ and $d_*'$ are neighbors means that for all valid $(i,j)$ we have $v_{i,j}\in \{\frac{n}{2}, \frac{n}{2}\pm 1\}$; it follows that $|v_{i, j_1} - v_{i, j_2}| \leq 2$, and the result follows for $n\geq 8$. The final sentence in the lemma follows from Lemma \ref{diff row properties} Rule 1.

Note that $2\equiv -4 \pmod{6}$, so for $n=6$ we may have that the difference of sums is $4$ without contradiction. Note that $|\epsilon_{i, j} - \epsilon_{i+1, j}| \leq 2$ always holds, so we must have a case like:

\begin{figure}[H]
\centering
\begin{tikzpicture}[scale=.6]
\draw(0,0)grid(3,2); 
\foreach\x[count=\i] in{1, \dots , 1}{\node at(\i-0.5,1.5){$\x$};};
\foreach\x[count=\i] in{-1, \dots, -1}{\node at(\i-0.5,0.5){$\x$};};
\end{tikzpicture}
\end{figure}
\noindent where the entries in between ($\dots$) are equal on both rows. I leave to the reader to verify this tiny case is always a contradiction. 
\end{proof}
A sneaky use of the Classification Theorem \ref{classification of paths} improves the result:
\begin{lemma}\label{1s above -1s}
Suppose that $d_* = (\epsilon_{1,1}, \epsilon_{1,2}, \dots, \epsilon_{1, n-1}, h)$ and $d_*' = (\epsilon_{2, 1}, \epsilon_{2,2}, \dots, \epsilon_{2, n-1}, h')$ are neighbors. There is no 1 above a -1 or vice versa, i.e. $\epsilon_{1,j}=\pm1\implies \epsilon_j'\neq \mp1$. Equivalently, the partial sums $\sum_{j=j_1}^{j_2} (\epsilon_{1, j} - \epsilon_{2,j})$ increment by $\pm 1$.
\end{lemma}
\begin{proof}
Suppose that for some column $j$, $\epsilon_{1,j} = 1$ and $\epsilon_{2,j} = -1$. If column $j+1$ exists, then $\epsilon_{1,j+1}\in \{0, 1\}$ and $\epsilon_{2, j+1}\in \{-1, 0\}$. If $\epsilon_{1,j+1} = 1$, no matter the value of $\epsilon_{2,j+1}$, this is a contradiction to Lemma \ref{adj sum lemma}. Likewise $\epsilon_{2,j+1}\neq -1$. We can conclude neighboring values are 0 on either side of $\epsilon_{1,j}$ and $\epsilon_{2,j}$. 

1) Suppose $1 < j < n-1$. Take for example $j=2$, so $d_* = (0, 1, 0, \epsilon_{1,4}, \dots)$ and $d_*' = (0, -1, 0, \epsilon_{2, 4}, \dots)$. There is no way to pick $\epsilon_{1,4}$ and $\epsilon_{2,4}$ in $\{0, \pm 1\}$ without a contradiction to Lemma \ref{adj sum lemma}, or Lemma \ref{disallowed patterns} Rule 2.

2) $j=1$ or $j=n-1$. Take the example $j=1$. Then we have $d_* = (1, 0, \epsilon_{1,3}, \dots)$ and $d_*' = (-1, 0, \epsilon_{2,3}, \dots)$. Neither epsilon term can be zero, so $\epsilon_{1,3} = \pm 1$. But there is no difference row beginning in $\pm(1, 0, 1, \dots)$ by Theorem \ref{classification of paths}! Therefore by Lemma \ref{disallowed patterns} we have $d_* = (1, 0, -1, -1, \dots)$ and 
 $d_*' = (-1, 0, 1, 1, \dots)$. Now columns 3 and 4 contradict Lemma \ref{adj sum lemma}.
\end{proof}

\begin{defn}
Suppose that $d_*$ and $d_*'$ are neighbors. If there exists a pair of columns $(j_1, j_2)$ such that:
\[\left|\sum_{j=j_1}^{j_2-1} (\epsilon_j' - \epsilon_j)\right| = 2\]
then $d_*$ and $d_*'$ are \textbf{determined} by columns $j_1, j_2$. 
\end{defn}
That is: if neighbors $r$ and $r'$ form a $2\times n$ Latin rectangle, there is a unique (i.e. determined) vertical difference $v_{1,j}$ for each column $j$. We will in fact show that every non-equal pair of neighbors is determined, but this nice behavior need not generalize to lower inner distances. The following lemma shows why this matters:

\begin{lemma}\label{unequal neighbors lemma}
Let $L\in \MiDn$. If $d_i$ and $d_{i+1}$ are determined by $j_1, j_2$, with $\epsilon_{i,j} = \epsilon_{i+1, j}$ for $j_1< j < j_2-1$, then row $i$ of the $V$ has the following entries from columns $j_1$ to $j_2$:
\begin{figure}[H]
    \centering
    \begin{tikzpicture}[scale=.6]
    \draw(0,0)grid(5,1); 
    \foreach\x[count=\i] in{1,0, \dots, 0, -1}{\node at(\i-0.5,0.5){$\x$};};
    \node at(0.5, 1.5) {$j_1$};
    \node at(4.5, 1.5) {$j_2$};
\end{tikzpicture}
\end{figure}
\noindent Further, $d_{i-1}$ and $d_{i+2}$ do not equal $d_{i}$ or $d_{i+1}$.
\end{lemma}
\begin{proof}
Assume there exists such $j_1, j_2$. WLOG, we are working with the following assumptions:
\[(h_{i,j_1} - h_{i+1, j_1}) = 1 = (h_{i,j_2-1} - h_{i+1, j_2-1})\qquad h_{i,j} = h_{i+1, j} \text{ for $j_1 < j < j_2-1$}\]
\[\implies v_{i, j_1} - v_{i, j_2} \equiv \sum_{j=j_1}^{j_2-1} (h_{i, j} - h_{i+1, j}) \equiv 2\]
Which immediately gives $v_{i, j_1} = \frac{n}{2} + 1$ and $v_{i, j_2} = \frac{n}{2}-1$ using Equation \ref{generalized condition 1}. The remaining $v_{i,j}$ for $j_1 < j < j_2-1$ follow from the same equation. We show $d_{i+2}\neq d_i$ and $d_{i+2}\neq d_{i+1}$; the same argument holds for $d_{i-1}$.

1) If $d_{i+2} = d_{i}$, the argument above shows that entries $j_1$ to $j_2$ on row $i+1$ of $V$ is of the same form as above, with $-1$ on the left and $1$ on the right instead. Since the columns of $V$ are difference rows, this is a contradiction to Lemma \ref{disallowed patterns} Rule 1. 

2) If $d_{i+2} = d_{i+1}$, then row $i+1$ of $V$ has a constant entry by Lemma \ref{row prod equivalent}. Since row $i$ contained $-1, 0$, and $1$, there is no choice of $v_{i, 1}$ without contradicting Lemma \ref{disallowed patterns} Rule 1.
\end{proof}

The following lemma establishes that type 1 neighboring pairs are determined:

\begin{lemma}\label{type 1 neighbors}
Suppose that $d_*$ is a type 1 difference row of length $n$ with $h$ satisfying $1 \leq h < \frac{n}{2} -1$. The neighbors of $d_*$ are the type 1 $d_*'$ with $h' \in \{h-2, h, h+2\}$, excluding Row $\pm$A, with $h'>0$. If $h = \frac{n}{2}-1$ and $d_*\neq$Row -A, this list includes Row A. Pairs of non-equal type 1 neighbors are determined.
\end{lemma}
\begin{proof}
Type 1's by definition contain only 0's and 1's (for $h > 0$. The negative case is identical). Since $d_*$ does not contain two 0's in a row, by Lemma \ref{1s above -1s}, a neighbor $d_*'$ does not contain two consecutive -1's. By Theorem \ref{classification of paths}, the only type 2's without two consecutive -1's have their -1's at entries 1 and/or $n-1$, whereas $d_*$ has 1's at these entries, so Lemma \ref{1s above -1s} gives that $d_*'$ is thus type 1. Lemma \ref{adj sum lemma} states $|h - h'| \leq 2$. If for example $h' = h+2$, then $j_1 = 1$ and $j_2 = n$ is a determined pair.

The same argument holds for $d_* = (1, 1, \underbrace{0, 1, \dots, 1, 0}_{\text{alternating}}, 1, 1, \frac{n}{2} - 1)$, so that all of its neighbors are type 1, except nothing rules out Row A from being a neighbor, and in fact they can be neighbors as one may check.
\end{proof}

\begin{thm}\label{row prod with type 1}
Suppose $L\in \MiDn$, and row $i$ of $L$ has a type 1 extended difference row $d_i\neq$Row A. Then $L$ is a product of rows.
\end{thm}
\begin{proof}
We prove first the case where $\epsilon_{i,n}\leq \frac{n}{2} - 3$, and leave the $\epsilon_{i,n}=\frac{n}{2}-1$ cases for later. We may assume $i\leq \frac{n}{2}$ by flipping the Latin square upside down.

Suppose $d_i$ is type 1 with $\epsilon_{i,n}\leq \frac{n}{2} - 3$. If $d_1 = d_2 = \dots = d_n$, we are done. Otherwise by Lemma \ref{type 1 neighbors}, if $d_{i+1} \neq d_i$ then $\epsilon_{i+1, n} = \epsilon_{i, n}\pm 2$ and $d_{i+1}$ is a type 1 other than Row A. Then Lemma \ref{type 1 neighbors} says $d_{i+2}$ is still a type 1, possibly Row A.

Assume WLOG $\epsilon_{i+1, n} = \epsilon_{i, n} + 2$. Observe that $d_i$ and $d_{i+1}$ have the same entries everywhere except columns $j_1 = \frac{n}{2} - \epsilon_{i, n}$ and $j_2 - 1 = \frac{n}{2} + \epsilon_{i, n}$, where $h_{i, j_1} = 1 = h_{i, j_2}$ and $h_{i+1, j_1} = 0 = h_{i+1, j_2}$. It follows from Lemma \ref{adj sum lemma} that row $i$ of $V$ is $(1, 1, \dots, 1, 0, \dots, 0, -1, \dots, -1)$: the center $j_2 - 1 - j_1\geq 1$ entries are zero, to the left is all 1's and the right is all -1's. In particular $v_{i, \frac{n}{2}}=0$. Repeating for $d_{i+1}$ and $d_{i+2}$ gives $v_{i+1, \frac{n}{2}} = 0$. This places two adjacent zeroes within column $\frac{n}{2}$ of $V$, a contradiction.

See the following example of neighboring pairs of type 1's with $\epsilon_{i, n}=1$ and $\epsilon_{i+1, n} = 3$ for $n = 8$.
\begin{figure}[H]
\centering
\begin{tikzpicture}[scale=.6]
    \draw(0,0)grid(8,2); 
    \draw[ultra thick](2, 0)rectangle(5, 2);
    \foreach\x[count=\i] in{ ,  ,j_1,  ,j_{2}-1,  ,  ,  }{\node at(\i-0.5,2.5){$\x$};};
    \foreach\x[count=\i] in{1, 1, 1, 1, 1, 1, 1, 1}{\node at(\i-0.5,1.5){$\x$};};
    \foreach\x[count=\i] in{1, 1, 0, 1, 0, 1, 1, 3}{\node at(\i-0.5,0.5){$\x$};};
    
    \node at(9.5, 1.5) {$\implies$};
    
    \draw(11,1)grid(19,2); 
    \draw[ultra thick](13,1)rectangle(17,2);
    \foreach\x[count=\i] in{,  ,j_1,  ,  ,j_2,   ,   }{\node at(\i+10.5,2.5){$\x$};};
    \foreach\x[count=\i] in{1, 1, 1, 0, 0, -1, -1, -1}{\node at(\i+10.5,1.5){$\x$};};
\end{tikzpicture}
\caption{Neighboring non-equal type 1 rows (left) give the resulting row of $V$ (right). Therefore there cannot exist a $3\times n$ rectangle with type 1 rows and maximum inner distance.}
\end{figure}

There is one more case to check: $d_i = (1,1, 0, \dots, 0, 1, 1, \frac{n}{2} - 1)$ with $d_{i+1} = (0, 1, 0, \dots, 0, 1, 0, 1 - \frac{n}{2}) =$ Row A. Lemma \ref{unequal neighbors lemma} says row $i$ of $V$ is exactly $(1, 0, \dots, 0, -1)$, all zeroes except the outer 2 entries. If all neighbors of Row A are determined, the latter case is proven. We verify this fact next. 
\end{proof}

Row A, $d_* = \pm(0, 1, 0, 1, \dots, 1, 0, 1-\frac{n}{2})$, has 5 neighbors in general: itself, the type 1 shown above with $h = \frac{n}{2} - 1$, and three type 2's. Two type 2 neighbor examples are shown below for $n=6$ and $n=8$:
\begin{figure}[H]
    \centering
\begin{tikzpicture}[scale=.6]
    \draw(0,0)grid(6,2); 
    \foreach\x[count=\i] in{1, 4, 2, 5, 3, 6}{\node at(\i-0.5,1.5){$\x$};};
    \foreach\x[count=\i] in{4, 1, 5, 3, 6, 2}{\node at(\i-0.5,0.5){$\x$};};
\end{tikzpicture}
\quad
\begin{tikzpicture}[scale=.6]
    \draw(0,0)grid(8,2); 
    \foreach\x[count=\i] in{1, 5, 2, 6, 3, 7, 4, 8}{\node at(\i-0.5,1.5){$\x$};};
    \foreach\x[count=\i] in{6, 1, 5, 2, 7, 4, 8, 3}{\node at(\i-0.5,0.5){$\x$};};
\end{tikzpicture}\\\vspace{0.1in}
\begin{tikzpicture}[scale=.6]
    \draw(0,0)grid(5,2); 
    \foreach\x[count=\i] in{0, 1, 0, 1, 0 }{\node at(\i-0.5,1.5){$\x$};};
    \foreach\x[count=\i] in{0, 1, 1, 0, -1}{\node at(\i-0.5,0.5){$\x$};};
\end{tikzpicture}
\qquad
\begin{tikzpicture}[scale=.6]
    \draw(0,0)grid(7,2); 
    \foreach\x[count=\i] in{0, 1, 0, 1, 0, 1, 0}{\node at(\i-0.5,1.5){$\x$};};
    \foreach\x[count=\i] in{-1, 0, 1, 1, 1, 0, -1}{\node at(\i-0.5,0.5){$\x$};};
\end{tikzpicture}
    \label{alt 0,1 counterex}
    \caption{Two Latin rectangles (top) of maximum inner distance and their horizontal difference matrices (bottom). The top difference rows are Row A, while the bottom rows are type 2.}
\end{figure}

Recall that type 2 rows $d_*$ depend only on two values, $(m,h)$, with $0\leq m\leq \frac{n}{2} - h - 1$ denoting the number of lead $\pm1$'s, and $1\leq h\leq\frac{n}{2} - 2$ the final entry in $d_*$.

\begin{lemma}\label{type 2 neighbors, Rows A and C}
\begin{enumerate} The following list describes all non-type 1 neighbors, all of which are determined:
    \item The neighbors of Row A are itself, the type 1 with $h = \frac{n}{2} - 1$, both type 2's with $h = 2 - \frac{n}{2}$, and the type 2 with $m = 1$ and $h = 3-\frac{n}{2}$.
    \item The neighbors of Row C are itself and its two rotations by $\pm 1$.
    \item Suppose $d_*$ is a type 2 extended difference row, determined by $(h,m)$. If $d_*'$ is a type 2 neighbor determined by $(h', m')$, then $|m - m'| \leq 1$.
\end{enumerate}
\end{lemma}
\begin{proof}
Throughout the proof $d_*, h$, and $m$ will denote the extended difference row, its final entry, and (if $d_*$ is a type 2) how many lead $\pm1$'s there are, respectively, for the row we wish to find the neighbors of. Similarly $d_*', h', m'$ denote the same for the rows we wish to be the neighbors of $d_*$.

[Proof of 1]: Let $d_* =$ Row A, so $h=1-\frac{n}{2}$. Lemma \ref{type 1 neighbors} says the type 1 row with $h' = \frac{n}{2} - 1$ is a neighbor of Row A. Since Row A contains only 0's and 1's, all of its neighbors have no consecutive -1's by Lemma \ref{adj sum lemma}. The only remaining rows that satisfy this are type 2's with $h' \in \{2 - \frac{n}{2}, 3 - \frac{n}{2}\}$. For $h' = 2-\frac{n}{2}$, $m'$ ranges in $[0,1]$, and both of these values work as one may check. For $h' = 3 - \frac{n}{2}$, $m'$ ranges in $[0, 2]$ but only $m' = 1$ works as the others contain 2 consecutive -1's. In each of these examples, the first two entries in Row A are $(0, 1)$, and the last two (excluding $h$) entries are $(1, 0)$. For the type 2 neighbors, either the first two entries are $(-1, 0)$, or the last two entries are $(0, -1)$, or both. In any case, these are determined. 

[Proof of 2]: Let $d_* =$Row C, so $h = 0$, and Lemma \ref{adj sum lemma} gives that $h'\in [-2, 2]$. $h'\in [-1,1]$ are the cycles. Row A and Row C are not neighbors by the proof of 1, so Lemma \ref{type 1 neighbors} says $d_*'$ must be a type 2. Row C starts with $\frac{n}{2} - 1 = m$ lead 1's; let $m'$ denote the number of lead 1's in $d_*'$. But Theorem \ref{classification of paths} says $m' \leq \frac{n}{2} - 1 - h'$, and following these 1's are 0, -1, -1. If $|m-m'| \geq 2$, this places a 1 above a -1, contradicting Lemma \ref{1s above -1s}. Therefore $|m-m'|\leq 1$. Therefore $d_*'$ is the rotation by $\pm 1$. To see these are determined, see the following example for $n=6$:
\begin{figure}[H]
    \centering
\begin{tikzpicture}[scale=.6]
    \draw(0,0)grid(6,2); 
    \draw[step = 2, ultra thick](1, 0)rectangle(3,2);
    \foreach\x[count=\i] in{1, 1, 0, -1, -1, 0}{\node at(\i-0.5,1.5){$\x$};};
    \foreach\x[count=\i] in{1, 0, -1, -1, 0, 1}{\node at(\i-0.5,0.5){$\x$};};
\end{tikzpicture}
\label{m - m' = 1}
\caption{Row C (top row) adjacent to its rotation by -1 (bottom row) for $n = 6$. If $m, m'$ denote the number of lead 1's, note that $m = 2$ and $m' = 1$, so that $|m - m'| = 1$, inducing a $(1,0)$ over a $(0,-1)$.}
\end{figure}

[Proof of 3]: We generalize the argument from the proof of 2: if $d_*$ and $d_*'$ start in say $m$ and $m'$ lead 1's, followed by 0 and -1 on both rows, then $|m - m'|\leq 1$ just as in the proof of Statement 1. For if $m \geq m' + 2$, Entries $m' + 1$ and $m' + 2$ of $d_*$ are $(1,1)$ while these entries in $d_*'$ are $(0, -1)$, contradicting \ref{adj sum lemma}. 

Specifically, since one of $h \neq h'$ or $m \neq m'$ is true, $d_*$ and $d_*'$ have either a different number of 1's at the start or a different number of 1's at the end. Where the difference occurs there is a pair of columns $j_1, j_2$ that make the pair determined. 
\end{proof}

\begin{thm}\label{determined theorem}
Every pair of non-equal neighbors is determined. Further, $L\in\MiDn$ is a row product if and only if there exists a pair of consecutive rows with the same difference row.
\end{thm}
\begin{proof}
Lemmas \ref{type 1 neighbors} and \ref{type 2 neighbors, Rows A and C} show that every pair of non-equal neighbors is determined.

The forward direction is by definition. The fact that the backwards direction holds is surprising and powerful. Lemma \ref{unequal neighbors lemma} shows that if $d_i = d_{i+1}$, then $d_{i+1}$ and $d_{i+2}$ cannot be a determined pair of neighbors, but now this means $d_{i+1} = d_{i+2}$. The same holds for row $i-1$, and every other row. Therefore $L$ is a row product.
\end{proof}

We have made great progress! Just a few more computations remain, hang in there.

\begin{lemma}\label{row prod with Row A}
Suppose $L\in \MiDn$. If $L$ contains Row A, then $L$ is a product of rows. 
\end{lemma}
\begin{proof}
The proof is showcased for $n = 12$, which is small enough to fit on paper and large enough to show the pattern that easily generalizes for $n > 12$. We show if $d_{i} =$Row A, then $d_{i+1}= d_i$.

Suppose $d_i =$Row A. Restrict $H$ to just rows $i, i+1$, and $i+2$, and $V$ to rows $i$ and $i+1$. By Lemma \ref{type 2 neighbors, Rows A and C}, the neighbors of $d_i$ are itself, a type 1, and 3 type 2's. We are assuming $d_{i+1}\neq d_i$, and Theorem \ref{row prod with type 1} says $d_2$ is not type 1.

The remaining three type 2 rows look exactly like Row A everywhere, except the first and last 3 entries may differ. The example for $n =12$ so far is shown below:
\begin{figure}[H]
\centering
\begin{tikzpicture}[scale=.6]
    \node at(5.5, 3.5) {$H$};
    \draw(0,0)grid(11,3); 
    \draw[ultra thick](0, 1)rectangle(3, 2);
    \draw[ultra thick](8, 1)rectangle(11, 2);\foreach\x[count=\i] in{$i$, $i+1$, $i+2$}{\node at(-0.6, 3.5-\i){$\x$};};
    \foreach\x[count=\i] in{0, 1, 0, 1, 0, 1, 0, 1, 0, 1, 0}{\node at(\i-0.5,2.5){$\x$};};
    \foreach\x[count=\i] in{*, *, *, 1, 0, 1, 0, 1, *, *, *}{\node at(\i-0.5,1.5){$\x$};};
    \foreach\x[count=\i] in{*, *, *, *, *, *, *, *, *, *, *}{\node at(\i-0.5,0.5){$\x$};};
\end{tikzpicture}
\quad
\begin{tikzpicture}[scale=.6]
    \node at(6, 2.5) {$V$};
    \draw(0,0)grid(12,2);
    \foreach\x[count=\i] in{$i$, $i+1$}{\node at(12.6, 2.5-\i){$\x$};};
    \foreach\x[count=\i] in{\cdot, \cdot, \cdot, 0, 0, 0, 0, 0, 0, \cdot, \cdot, \cdot,}{\node at(\i-0.5,1.5){$\x$};};
    \foreach\x[count=\i] in{\cdot, \cdot, \cdot, \pm1, \pm1, \pm1, \pm1, \pm1, \pm1, \cdot, \cdot, \cdot,}{\node at(\i-0.5,0.5){$\x$};};
\end{tikzpicture}
    \caption{The first 2 rows of $H$ (left) induce the following entries in $V$ (right).}
\end{figure}
Either the left bolded $(*, *, *)$ equals $(-1, 0, 1)$ or the right bolded $(*, *, *)$ equals $(1, 0, -1)$, or both, depending on which type 2 neighbor we choose for $d_{i+1}$, otherwise the bolded entries are equal those above. Every scenario implies the center $n - 6$ entries in row $i$ of $V$ are 0. $V$ is a $2\times n$ subset of a difference matrix, so the entries directly below 0's are non-zero. Further, they are all equal since Lemma \ref{1s above -1s} translates for columns as 1's and -1's are not horizontally adjacent in $V$. 

Since the middle entries of row $i+1$ of $V$ are constant, the middle entries in $d_{i+2}$ must match those in $d_{i+1}$. We know $d_{i+2}=$ Row A contradicts the second half of Lemma \ref{adj sum lemma}, so $d_{i+2}$ is a type 2. Assume WLOG that the first bolded $(*, *, *)$ entries in $d_{i+2}$ is $(-1, 0, 1)$, for illustrative purposes:
\begin{figure}[H]
\centering
\begin{tikzpicture}[scale=.6]
    \node at(5.5, 3.5) {$H$};
    \draw(0,0)grid(11,3); 
    \draw[ultra thick](0, 0)rectangle(3, 1);
    \foreach\x[count=\i] in{$i$, $i+1$, $i+2$}{\node at(-0.6, 3.5-\i){$\x$};};
    \foreach\x[count=\i] in{0, 1, 0, 1, 0, 1, 0, 1, 0, 1, 0}{\node at(\i-0.5,2.5){$\x$};};
    \foreach\x[count=\i] in{-1,0, 1, 1, 0, 1, 0, 1, *, *, *}{\node at(\i-0.5,1.5){$\x$};};
    \foreach\x[count=\i] in{*, *, *, 1, 0, 1, 0, 1, *, *, *}{\node at(\i-0.5,0.5){$\x$};};
\end{tikzpicture}
\quad
\begin{tikzpicture}[scale=.6]
    \node at(6, 2.5) {$V$};
    \draw(0,0)grid(12,2); 
    \draw[ultra thick](0, 0)rectangle(3, 1);
    \foreach\x[count=\i] in{$i$, $i+1$}{\node at(12.6, 2.5-\i){$\x$};};
    \foreach\x[count=\i] in{1, 0, -1, 0, 0, 0, 0, 0, 0, \cdot, \cdot, \cdot,}{\node at(\i-0.5,1.5){$\x$};};
    \foreach\x[count=\i] in{\cdot, \cdot, \cdot, \pm1, \pm1, \pm1, \pm1, \pm1, \pm1, \cdot, \cdot, \cdot,}{\node at(\i-0.5,0.5){$\x$};};
\end{tikzpicture}
    \caption{First 3 rows of $H$, under the assumption that row 2 begins in (-1, 0, 1).}
\end{figure}
Two cases follow from Lemma \ref{type 2 neighbors, Rows A and C} Statement 3: (1) if the bolded $(*, *, *)$ is $(-1, 0, 1)$, then the bolded $(\cdot, \cdot, \cdot)$ must be $(\pm 1, \pm 1, \pm 1)$, which places a 1 below a -1 or a -1 below a 1. (2) If the bolded $(*, *, *)$ is $(0, 1, 0)$, then the bolded $(\cdot, \cdot, \cdot)$ must be $(-1, 0, 1)$, which does the same. Both cases are a contradiction, and a similar argument happens if we had chosen the last 3 entries in $d_2$ to be $(1, 0, -1)$. 
\end{proof}

Finally, we can say the same with type 2's.
\begin{thm}\label{row prod with type 2}
If $L$ contains Row C, then $L$ is a circulant (or back circulant) Latin square or a row product. If $L$ contains a type 2 row that is not cyclic, then $L$ is a row product. 
\end{thm}
\begin{proof}
First we show that if $L$ contains Row C followed by a non-equal neighbor, then $L$ is circulant (or back circulant, and contains only cycles) or a row product. Then we show if $L$ is not a row product and contains a type 2 row $d_*$, then $L$ must contains Row C, a contradiction.

Suppose that $d_i=$Row C. Then WLOG choose its neighbor $d_{i+1}$ to be a rotation of Row C by -1. Once again we use the example $n=12$:
\begin{figure}[H]
\centering
\begin{tikzpicture}[scale=.6]
    \draw(0,0)grid(11,3); 
    \node at(5.5, 3.5) {$H$};
    \foreach\x[count=\i] in{$i$, $i+1$, $i+2$}{\node at(-0.6, 3.5-\i){$\x$};};
    \foreach\x[count=\i] in{1, 1, 1, 1, 1, 0, -1, -1, -1, -1, -1}{\node at(\i-0.5,2.5){$\x$};};
    \foreach\x[count=\i] in{1, 1, 1, 1, 0, -1, -1, -1, -1, -1, 0}{\node at(\i-0.5,1.5){$\x$};};
    \foreach\x[count=\i] in{*, *, *, *, *, *, *, *, *, *, *}{\node at(\i-0.5,0.5){$\x$};};
\end{tikzpicture}
\quad
\begin{tikzpicture}[scale=.6]
    \draw(0,0)grid(12,2); 
    \node at(5.5, 2.5) {$V$};
    \foreach\x[count=\i] in{$i$, $i+1$}{\node at(12.6, 2.5-\i){$\x$};};
    \foreach\x[count=\i] in{1, 1, 1, 1, 1, 0, -1, -1, -1, -1, -1, 0}{\node at(\i-0.5,1.5){$\x$};};
    \foreach\x[count=\i] in{\cdot, \cdot, \cdot, \cdot, \cdot, \cdot, \cdot, \cdot, \cdot, \cdot, \cdot, \cdot}{\node at(\i-0.5,0.5){$\x$};};
\end{tikzpicture}
\end{figure}
Observe that in $H$, the first $\frac{n}{2} - 1$ entries in row $i$ are 1, while in row $i+1$ the first $\frac{n}{2} - 2$ entries are 1. We show this implies row $i+2$ has $\frac{n}{2} - 3$.

Suppose $d_{i+i}$ has $m$ lead 1's, followed by 0 and $m+1$ many -1's. If $d_{i+2}$ has $m$ lead 1's as well, then the first $2m+2$ entries in $d_{i+1}$ and $d_{i+2}$ coincide, thus the first $2m+3$ entries in row $i+1$ of $V$ are constant. Since entries $m+1, m+2, m+3$ in row $i$ of $V$ are 1, 0, and -1, this is a contradiction. Therefore row $i+2$ has $m\pm 1$ lead 1's by Lemma \ref{type 2 neighbors, Rows A and C} Statement 3. Since row $i$ had $m+1$ lead ones, one can check that if row $i+2$ does as well, then entries $m+1, m+2, m+3$ of row $i+1$ of $V$ are $-1, 0, 1$, another contradiction.

The argument above holds whenever $1\leq m \leq \frac{n}{2} - 2$. Further, the same argument holds in reverse: in general as $i$ increases the number of 1's at the front or ends of rows in $H$ must stay the same, increase by 1 until reaching a maximum of $\frac{n}{2} - 1$, or decrease by 1 until reaching the minimum of 0. Returning to our example, row $i+2$ of $H$ must then be of the following form:
\begin{figure}[H]
\centering
\begin{tikzpicture}[scale=.6]
    \draw(0,0)grid(11,3); 
    \draw[ultra thick](8,0)rectangle(11,1);
    \node at(5.5, 3.5) {$H$};
    \foreach\x[count=\i] in{$i$, $i+1$, $i+2$}{\node at(-0.6, 3.5-\i){$\x$};};
    \foreach\x[count=\i] in{1, 1, 1, 1, 1, 0, -1, -1, -1, -1, -1}{\node at(\i-0.5,2.5){$\x$};};
    \foreach\x[count=\i] in{1, 1, 1, 1, 0, -1, -1, -1, -1, -1, 0}{\node at(\i-0.5,1.5){$\x$};};
    \foreach\x[count=\i] in{1, 1, 1, 0, -1, -1, -1, -1, *, *, *}{\node at(\i-0.5,0.5){$\x$};};
\end{tikzpicture}
\quad
\begin{tikzpicture}[scale=.6]
    \draw(0,0)grid(12,2); 
    \node at(5.5, 2.5) {$V$};
    \foreach\x[count=\i] in{$i$, $i+1$}{\node at(12.6, 2.5-\i){$\x$};};
    \foreach\x[count=\i] in{1, 1, 1, 1, 1, 0, -1, -1, -1, -1, -1, 0}{\node at(\i-0.5,1.5){$\x$};};
    \foreach\x[count=\i] in{1, 1, 1, 1, 0, -1, -1, -1, -1, \cdot, \cdot, \cdot}{\node at(\i-0.5,0.5){$\x$};};
\end{tikzpicture}
\end{figure}
The only possible ways to fill the bolded $(*,*,*)$ is $(0, -1, 0)$ or $(-1, 0, 1)$. But the first case places a 0 at the far right entry in row $i+1$ of $V$, a contradiction. It follows that row $i+2$ is the rotation by -2 of Row C. Now using the property above that the number of lead or end 1's must constantly increase and decrease, the arguments continues until every row of $H$ is a cycle rotated by $\pm1$ from the previous. Therefore $L$ is back-circulant if the rows rotate by $-1$, or circulant if they rotate by 1.

We have further shown that if $L$ contains a cycle, then $L$ is a circulant or back-circulant. Now we can assume $d_1$ is a type 2 non-cyclic row. As we increase $i$, the amount of 1's at the beginnings and ends of $d_i$ must constantly increase, decrease, or stay fixed. The total number of 1's in an extended difference row $d_i$ depends only on the magnitude of its final entry $\epsilon_{i,n}$, so as $i$ increases, $|\epsilon_{i,n}|$ must increase or decrease at a constant rate, or stay fixed.

1) $\epsilon_{i,n}$ values are capped by $|\epsilon_{i,n}| \leq\frac{n}{2} -2$, so if $|\epsilon_{i,n}|$ increases with $i$, we can fill at most $\frac{n}{2} - 2$ rows before reaching the maximum. The only rows with larger $\epsilon_{i,n}$ value are type 1.

2) $|\epsilon_{i,n}|=0$ or $|\epsilon_{i,n}|=1$ correspond to cycles, while Lemma \ref{adj sum lemma} says $|\epsilon_{i,n} - \epsilon_{i+1, n}|\leq 2$. Therefore starting at $\epsilon_{1,n} \geq 2$ and decreasing, $L$ must contain a cycle. 

3) $|\epsilon_{i,n}|$ cannot stay fixed indefinitely, as for any $\epsilon_{i,n}$ there are exactly $\frac{n}{2} - \epsilon_{i,n}$ extended difference rows ending with $\epsilon_{i,n}$.

None of the three cases fill out an entire Latin square $L$. Therefore, if $L$ contains a non-cyclic type 2 row, then $L$ is a row product
\end{proof}

\begin{thm}\label{main result}
Every Latin square of even order $n\geq 6$ with maximum inner distance is a circulant or back-circulant Latin square, or is a product of rows of maximum inner distance. The number of Latin squares of maximum inner distance is given by $n(P(n)^2 + 2n)$, where:
\[P(n) = 
\begin{cases}
    \frac{1}{4}n^2 + 2 & \text{if $n = 4k$} \\
    \frac{1}{4}n^2 + 1 & \text{if $n = 4k + 2$}
\end{cases}\]
is the number of rows of maximum inner distance beginning in the symbol 1.
\end{thm}
\begin{proof}
Theorems \ref{classification of paths} establishes that every row of maximum inner distance can be classified as a type 1 row, type 2 row, or is the cyclic row known as Row C. Some type 2 rows are \textit{cyclic rows}. Theorem \ref{number of paths} counts this number as $P(n)$ given in the theorem statement above. Lemma \ref{determined theorem} showed that any pair of non-equal neighboring rows is \textit{determined}, which allowed lemmas \ref{row prod with type 1} and \ref{row prod with Row A} to show that if a Latin square of maximum inner distance $L$ contained a type 1, then $L$ is a row product. Theorem \ref{row prod with type 2} did the same for non-cyclic type 2's, and further if $L$ contains a cyclic row, it is a row product or a circulant or back-circulant.
\end{proof}

\section{Acknowledgements and Conclusion}
This project was done as part of an honors thesis at the University of Texas at Austin. The combinatorial topic of Latin squares has been of interest to me since the Summer of 2020, where I met my first advisor James Hammer and as a group developed the concept of inner distance\cite{inner distance}. Despite the results in this paper being mostly elementary computations, the results quickly increase in complexity as inner distance decreases with large values of $n$, making it a surprise that it is possible to compute in general for the maximum inner distance. My future hopes are to find some correspondence of inner distance in Latin squares with another area of mathematics, for example how the multiplication tables of groups or quasigroups are Latin squares, as well as to expand the classification of large inner distance squares to other types of squares such as Sudoku Latin squares or to odd ordered squares of inner distance $\frac{n-3}{2}$. Lastly, I want to deeply thank my advisors James Hammer and Lorenzo Sadun, for supporting me in this project.


\begin{thebibliography}{99}
\bibitem{inner distance}
Omar Aceval, Paige Beidelman, Jieqi Di, James Hammer, Mitchel O'Connor, Caitlin Owens and Yewen Sun.
\newblock Distance in Latin Squares, 2021;
\newblock arXiv:2107.06437.

\bibitem{number of LS}
Alexander Hulpke, Petteri Kaski and Patric R. J. Östergård.
\newblock The number of Latin squares of order 11, 2009;
\newblock arXiv:0909.3402.

\bibitem{oeis 4k+2}
OEIS Foundation Inc., The On-Line Encyclopedia of Integer Sequences;
\newblock http://oeis.org/A069894.

\bibitem{oeis 4k}
OEIS Foundation Inc., The On-Line Encyclopedia of Integer Sequences;
\newblock http://oeis.org/A005899.

\bibitem{oeis evens}
OEIS Foundation Inc., The On-Line Encyclopedia of Integer Sequences;
\newblock http://oeis.org/A248800.


\end{thebibliography}
\end{document}